\documentclass{amsart}
\usepackage{amsfonts,amsthm,amsmath}
\usepackage{amssymb}
\usepackage{amscd}
\usepackage[utf8]{inputenc}
\usepackage[a4paper]{geometry}
\usepackage{enumerate}
\usepackage[usenames,dvipsnames]{color}
\usepackage{array}
\usepackage{array,tabularx}

\DeclareMathOperator{\Coeff}{Coeff}
\DeclareMathOperator*{\Res}{Res}
\newcommand\note[1]{\mbox{}\marginpar{ \scriptsize\raggedright
\hspace{1pt}\color{red} #1}}

\numberwithin{equation}{section}
\numberwithin{equation}{subsection}

\theoremstyle{plain}

\newtheorem{theorem}[equation]{Theorem}
\newtheorem{lemma}[equation]{Lemma}
\newtheorem{proposition}[equation]{Proposition}

\newtheorem{corollary}[equation]{Corollary}

\newtheorem{thm}[equation]{Theorem}
\newtheorem{cor}[equation]{Corollary}

\theoremstyle{definition}
\newtheorem{example}[equation]{Example}
\newtheorem{remark}[equation]{Remark}

\newtheorem{definition}[equation]{Definition}

\newtheorem{bekezdes}[equation]{}

\def\C{\mathbb C}
\def\Q{\mathbb Q}

\def\Z{\mathbb Z}

\def\im{{\rm Im}}

\newcommand{\calv}{{\mathcal V}}

\newcommand{\calt}{{\mathcal T}}

\newcommand{\cali}{{\mathcal I}}

\newcommand{\calO}{{\mathcal O}}

\newcommand{\calS}{{\mathcal S}}

\newcommand{\calL}{\mathcal{L}}

\newcommand{\tX}{\widetilde{X}}

\newcommand{\cX}{{\mathcal X}}
\newcommand{\cO}{{\mathcal O}}
\newcommand{\bP}{{\mathbb P}}
\newcommand*{\linebundle}{\mathcal{L}}
\newcommand{\bC}{{\mathbb C}}
\newcommand{\cF}{{\mathcal F}}

\newcommand{\eca}{{\rm ECa}}
\newcommand{\pic}{{\rm Pic}}

\newcommand{\m}{\mathfrak{m}}\newcommand{\fr}{\mathfrak{r}}
\newcommand{\mfl}{\mathfrak{L}}

\newcommand{\bt}{{\mathbf t}}

\newcommand{\bZ}{{\mathbb{Z}}}
\newcommand{\bQ}{{\mathbb{Q}}}

\author{J\'anos Nagy}
\address{Central European University, Dept. of Mathematics,  Budapest, Hungary}
\email{nagy\textunderscore janos@phd.ceu.edu}

\title{Invariants of relatively generic surface singularities  II. Images of Abel maps}

\begin{document}



\begin{abstract}

In \cite{R} the author investigated invariants of relatively generic structures on surface singularities generalising results of \cite{NNA1} and \cite{NNA2} about generic analytic structures and generic line bundles to the case of the relative setup, where we fix a given analytic type or line bundle on a smaller subgraph or more generally
on a smaller cycle and we choose a relatively generic line bundle or analytic type on the large cycle and managed to compute some of it's invariants, like geometric genus or $h^1$ of natural
line bundles.

In \cite{NNAD} the authors investigated the images of Abel maps $c^{l'}(Z) : \eca^{l'}(Z) \to \pic^{l'}(Z)$, where $l' \in - S'(|Z|)$, especially the dimensions of the images of
these maps and gave two algorithms to compute these invariants from cohomology numbers of cycles and from periodic constants of singularities we get from $\tX$ by blowing it up
at generic points sequentially.
Furthemore in \cite{NNAD} the authors gave explicit combinatorial formulas in the case of generic singularities.

In this paper we want to generalise the theorems from \cite{NNAD} to the relatively generic case.
In this case we fix a subsingularity $\tX_1$ for a subgraph $\mathcal{T}_1 \subset \mathcal{T}$ and a relatively generic singularity $\tX$ corresponding to $\tX_1$. 
Furthermore we fix a line bundle $\calL$ on $\tX_1$ and a Chern class $l' \in - S'$, such that $c^1(\calL)  = R(l')$.
Our main goal in the article is to compute $\dim(c^{l'}(\eca^{ l', \calL}(Z)))$ from invariants of the subsingularity $\tX_1$ and to conclude a few corollaries.

\end{abstract}

\maketitle

\linespread{1.2}


\pagestyle{myheadings} \markboth{{\normalsize  J. Nagy}} {{\normalsize Line bundles}}


\section{Introduction}\label{s:intr}

In \cite{R} the author investigated invariants of relatively generic structures on surface singularities generalising results of \cite{NNA1} and \cite{NNA2} about generic analytic structures and generic line bundles to the case of the relative setup, where we fix a given analytic type or line bundle on a smaller subgraph or more generally
on a smaller cycle and we choose a relatively generic line bundle or analytic type on the large cycle.

In \cite{NNAD} the authors investigated the images of Abel maps $c^{l'}(Z) : \eca^{l'}(Z) \to \pic^{l'}(Z)$, where $l' \in - S'(|Z|)$, especially the dimension of the images of
these maps and gave two algorithms to compute these invariants from cohomology numbers of cycles and from periodic constants of singularities we get from $\tX$ by blowing it up
at generic points sequentially, namely one of the main theorems was the following:

\begin{theorem}\label{dim}
Let's fix an arbitrary singularity with resolution $\tX$, a Chern class $l' \in - S'$ and an effective cycle $Z\geq E$, then one has:
\begin{equation}  \label{eq:form3}
\dim(\im(c^{l'}(Z))) = \min_{0\leq Z_1 \leq Z}\{\, (l', Z_1) + h^1(\calO_Z) - h^1(\calO_{Z_1})\, \}.
\end{equation}
\end{theorem}

Furthemore in \cite{NNAD} the authors gave explicit combinatorial formulas in the case of generic singularities:

\begin{cor} 
Assume that $\tX$ has a generic analytic type corresponding to a fixed resolution graph $\mathcal{T}$, $Z \geq E$ an integral  cycle and  $l' \in -S'$.
For any $0\leq Z_1\leq Z$  write $E_{|Z_1|}$ for $\sum _{E_v\subset |Z_1|}E_v$. Then
\begin{equation}\label{eq:gen}
\dim(\im(c^{l'}(Z))) = 1  -
 \min_{E \leq l \leq Z}\ \{ \chi(l)\} +  \min_{0\leq Z_1\leq Z}\big\{ \, (l', Z_1)  +  \min_{E_{|Z_1|} \leq l \leq Z_1 }(\chi(l))
  - \chi(E_{|Z_1|})
 \, \big\}.
\end{equation}
In particular, $\dim(\im(c^{l'}(Z)))$ is topological.
\end{cor}

In the present article we wish to generalise the theorems from \cite{NNAD} to the relatively generic case.
In this case we fix a subsingularity $\tX_1$ for a subgraph $\mathcal{T}_1 \subset \mathcal{T}$ and a relatively generic singularity $\tX$ corresponding to $\tX_1$. 
Furthermore we fix a line bundle $\calL$ on $\tX_1$ and a Chern class $l' \in - S'$, such that $c^1(\calL)  = R(l')$.

Let's denote the subset of effective Cartier divisors $\eca^{ l', \calL}(Z) \subset \eca^{ l'}(Z)$, such that $D \in \eca^{ l', \calL}(Z)$ if and only if
$\calO_{Z_1}(D) = \calL | Z_1$.

By \cite{R} we know, that $\eca^{ l', \calL}(Z)$ is an irreducible smooth subvariety of $\eca^{ l'}(Z)$.

Notice that $c^{l'}(\eca^{ l', \calL}(Z)) = c^{l'}(\eca^{ l'}(Z)) \cap r^{-1}(\calL |Z_1)$.

Our main goal in the article is to compute $\dim(c^{l'}(\eca^{ l', \calL}(Z)))$ from topological and analytical invariants of the subsingularity $\tX_1$ and to conclude a few corollaries.

In section 2) we summarise the topological and analytic invariants we need in the article.

In section 3) we summarise the results we need from \cite{R} about cohomology of relatively generic line bundles and cohomology of natural line bundles on relatively generic
singularities.

In section 4) we state and proof our main theorem about dimensions of images of relative Abel maps.

In section 5) as a corollary we give a formula for dimensions of Abel images  on relatively generic singularities, which is different in structure then the formulae given by the results
from \cite{NNAD}.

\section{Preliminaries}\label{s:prel}

\subsection{The resolution}\label{ss:notation}
Let $(X,o)$ be the germ of a complex analytic normal surface singularity,
 and let us fix  a good resolution  $\phi:\widetilde{X}\to X$ of $(X,o)$.
We denote the exceptional curve $\phi^{-1}(0)$ by $E$, and let $\cup_{v\in\calv}E_v$ be
its irreducible components. Set also $E_I:=\sum_{v\in I}E_v$ for any subset $I\subset \calv$.
For the cycle $l=\sum n_vE_v$ let its support be $|l|=\cup_{n_v\not=0}E_v$.
For more details see \cite{trieste,NCL,Nfive}.
\subsection{Topological invariants}\label{ss:topol}
Let $\calt$ be the dual resolution graph
associated with $\phi$;  it  is a connected graph.
Then $M:=\partial \widetilde{X}$ can be identified with the link of $(X,o)$, it is also
an oriented  plumbed 3--manifold associated with $\calt$.
We will assume that  $M$ is a rational homology sphere,
or, equivalently,  $\mathcal{T}$ is a tree and all genus
decorations of $\mathcal{T}$ are zero. We use the same
notation $\mathcal{V}$ for the set of vertices, and $\delta_v$ for the valency of a vertex $v$.

$L:=H_2(\widetilde{X},\mathbb{Z})$, endowed
with a negative definite intersection form  $I=(\,,\,)$, is a lattice. It is
freely generated by the classes of 2--spheres $\{E_v\}_{v\in\mathcal{V}}$.
 The dual lattice $L':=H^2(\widetilde{X},\mathbb{Z})$ is generated
by the (anti)dual classes $\{E^*_v\}_{v\in\mathcal{V}}$ defined
by $(E^{*}_{v},E_{w})=-\delta_{vw}$, the opposite of the Kronecker symbol.
The intersection form embeds $L$ into $L'$. Then $H_1(M,\mathbb{Z})\simeq L'/L$, abridged by $H$.
Usually one also identifies $L'$ with those rational cycles $l'\in L\otimes \Q$ for which
$(l',L)\in\Z$, or, $L'={\rm Hom}_\Z(L,\Z)$.

Each class $h\in H=L'/L$ has a unique representative $r_h=\sum_vr_vE_v\in L'$ in the semi-open cube
(i.e. each $r_v\in \bQ\cap [0,1)$), such that its class  $[r_h]$ is $h$.

All the $E_v$--coordinates of any $E^*_u$ are strict positive.
We define the Lipman cone as $\calS':=\{l'\in L'\,:\, (l', E_v)\leq 0 \ \mbox{for all $v$}\}$.
It is generated over $\bZ_{\geq 0}$ by $\{E^*_v\}_v$.

\subsection{Analytic invariants}\label{ss:analinv}

The group ${\rm Pic}(\widetilde{X})$ of  isomorphism classes of analytic line bundles on $\widetilde{X}$ appears in the exact sequence
\begin{equation}\label{eq:PIC}
0\to {\rm Pic}^0(\widetilde{X})\to {\rm Pic}(\widetilde{X})\stackrel{c_1}
{\longrightarrow} L'\to 0, \end{equation}
where  $c_1$ denotes the first Chern class.

 Here $ {\rm Pic}^0(\widetilde{X})=H^1(\widetilde{X},\calO_{\widetilde{X}})\simeq
\C^{p_g}$, where $p_g$ is the {\it geometric genus} of
$(X,o)$. $(X,o)$ is called {\it rational} if $p_g(X,o)=0$.
 Artin in \cite{Artin62,Artin66} characterized rationality topologically
via the graphs; such graphs are called `rational'. By this criterion, $\calt$
is rational if and only if $\chi(l)\geq 1$ for any effective non--zero cycle $l\in L_{>0}$.
Here $\chi(l)=-(l,l-Z_K)/2$, where $Z_K\in L'$ is the (anti)canonical cycle
identified by adjunction formulae
$(-Z_K+E_v,E_v)+2=0$ for all $v$.

The epimorphism
$c_1$ admits a unique group homomorphism section $l'\mapsto s(l')\in {\rm Pic}(\widetilde{X})$,
 which extends the natural
section $l\mapsto \calO_{\widetilde{X}}(l)$ valid for integral cycles $l\in L$, and
such that $c_1(s(l'))=l'$  \cite{trieste,OkumaRat}.
We call $s(l')$ the  {\it natural line bundles} on $\widetilde{X}$.
By  their definition, $\calL$ is natural if and only if some power $\calL^{\otimes n}$
of it has the form $\calO(-l)$ for some $l\in L$.

\bekezdes
Similarly, if $Z\in L_{>0}$ is an effective non--zero integral cycle such that $|Z| =E$, and $\calO_Z^*$ denotes
the sheaf of units of $\calO_Z$, then ${\rm Pic}(Z)=H^1(Z,\calO_Z^*)$ is  the group of isomorphism classes
of invertible sheaves on $Z$. It appears in the exact sequence
  \begin{equation}\label{eq:PICZ}
0\to {\rm Pic}^0(Z)\to {\rm Pic}(Z)\stackrel{c_1}
{\longrightarrow} L'\to 0, \end{equation}
where ${\rm Pic}^0(Z)=H^1(Z,\calO_Z)$.
If $Z_2\geq Z_1$ then there are natural restriction maps (for simplicity we denote all of them by
the same symbol $r$), ${\rm Pic}(\widetilde{X})\to {\rm Pic}(Z_2)\to {\rm Pic}(Z_1)$. Similar restrictions are defined at  ${\rm Pic}^0$ level too.
These restrictions are homomorphisms of the exact sequences  (\ref{eq:PIC}) and (\ref{eq:PICZ}).

Furthermore, we define a section of (\ref{eq:PICZ}) by
$s_Z(l'):=r(s(l'))={\mathcal O}_{\widetilde{X}}(l')|_{Z}$.
They also satisfies $c_1\circ s_Z={\rm id}_{L'}$. We write  ${\mathcal O}_{Z}(l')$ for $s_Z(l')$, and we call them
 {\it natural line bundles } on $Z$.

We also use the notations ${\rm Pic}^{l'}(\widetilde{X}):=c_1^{-1}(l')
\subset {\rm Pic}(\widetilde{X})$ and
${\rm Pic}^{l'}(Z):=c_1^{-1}(l')\subset{\rm Pic}(Z)$
respectively. Multiplication by $\calO_{\widetilde{X}}(-l')$, or by
$\calO_Z(-l')$, provides natural vector space isomorphisms
${\rm Pic}^{l'}(\widetilde{X})\to {\rm Pic}^0(\widetilde{X})$ and
${\rm Pic}^{l'}(Z)\to {\rm Pic}^0(Z)$.

Consider an effective cycle $Z $, and a Chern class $l'\in-\calS'$ associated with a resolution $\tX$, as above. 

Then, besides the Abel map $c^{l'}(Z)$ one can consider its `multiples' $\{c^{nl'}(Z)\}_{n\geq 1}$.

 It turns out that $n\mapsto \dim \im (c^{nl'}(Z))$
is a non-decreasing sequence,  $\im (c^{nl'}(Z))$ is an affine subspace
for $n\gg 1$, whose dimension $e_Z(l')$ is independent of $n\gg 0$, and essentially it depends only
on the $E^*$--support of $l'$ (i.e., on $I\subset \calv$, where $-l'=\sum_{v\in I}a_vE^*_v$ with all
$\{a_v\}_{v\in I}$ nonzero).
The statement  $e_Z(l')=e_Z(I)$ plays a crucial role in different analytic properties of $\tX$
(surgery formula, $h^1(\calL)$--computations, base point freeness properties).
 For details and for more about effective Cartier divisors and Abel maps see \cite{NNA1}.

\subsection{Notations.} We will write $Z_{min}\in L$ for the  {\it minimal} (or fundamental, or Artin) cycle, which is
the minimal non--zero cycle of $\calS'\cap L$ \cite{Artin62,Artin66}. Yau's {\it maximal ideal cycle}
$Z_{max}\in L$ defines the  divisorial part of the pullback of the maximal ideal $\m_{X,o}\subset \calO_{X,o}$, i.e.
 $\phi^*{\m_{X,o}}\cdot \calO_{\widetilde{X}}=\calO_{\widetilde{X}}(-Z_{max})\cdot \cali$,
where $\cali$ is an ideal sheaf with 0--dimensional support \cite{Yau1}. In general $Z_{min}\leq Z_{max}$.

\section{Relatively generic analytic structures on surface singularities}

In this section we wish to summarise the results from \cite{R} about relatively generic analytic structures we need in this article. 

\subsection{The relative setup.}

We consider a cycle $Z \geq E$ on the resolution $\tX$, and a smaller cycle $Z_1 \leq Z$, where we denote $|Z_1| = \calv_1$ and the subgraph corresponding to it by $\Gamma_1$.
We have the restriction map $r:\pic(Z)\to \pic(Z_1)$ and one has also the  (cohomological) restriction operator
  $R_1 : L'(\Gamma) \to L_1':=L'(\Gamma_1)$
(defined as $R_1(E^*_v(\Gamma))=E^*_v(\Gamma_1)$ if $v\in \calv_1$, and
$R_1(E^*_v(\Gamma))=0$ otherwise).
For any $\calL\in \pic(Z)$ and any $l'\in L'(\Gamma)$ it satisfies
\begin{equation}
c_1(r(\calL))=R_1(c_1(\calL)).
\end{equation}

In particular,
we have the following commutative diagram as well:

\begin{equation*}  
\begin{picture}(200,40)(30,0)
\put(50,37){\makebox(0,0)[l]{$
\ \ \eca^{l'}(Z)\ \ \ \ \ \stackrel{c^{l'}}{\longrightarrow} \ \ \ \pic^{l'}(Z)$}}
\put(50,8){\makebox(0,0)[l]{$
\eca^{R_1(l')}(Z_1)\ \ \stackrel{c^{R_1(l')}}{\longrightarrow} \  \pic^{R_1(l')}(Z_1)$}}
\put(162,22){\makebox(0,0){$\downarrow \, $\tiny{$r$}}}
\put(78,22){\makebox(0,0){$\downarrow \, $\tiny{$\fr$}}}
\end{picture}
\end{equation*}

By the `relative case' we mean that instead of the `total' Abel map
$c^{l'}$ (with $l'\in -\calS'$ and  $Z\geq E$)
we study its restriction above a fixed fiber of $r$.
That is, we fix some  $\mfl\in \pic^{R_1(l')}(Z_1)$, and we study
the restriction of $c^{l'}$ to $(r\circ c^{l'})^{-1}(\mfl)\to r^{-1}(\mfl)$.

\begin{definition} Denote the subvariety $(r\circ c^{l'})^{-1}(\mfl)
=(c^{R_1(l')} \circ \fr)^{-1}(\mfl) \subset \eca^{l'}(Z)$ by
$\eca^{l', \mfl}$.
\end{definition}

\begin{proposition}\label{lem:locsubmer} (a) $\fr$ is a local submersion, that is, for any
$D\in \eca ^{l'}(Z)$ and $D_1:=\fr(D)$, the tangent map $T_D\fr$ is surjective.

(b) $\fr$ is dominant.

(c) any non--empty fiber of  $\fr$ is smooth of dimension
$(l', Z)-(l',Z_1)=(l',Z_2)$, and it is irreducible.

\end{proposition}

\begin{corollary}\label{cor:smoothirreddim}
Fix $l'\in -\calS'$, $Z\geq E$, $Z_1 \leq Z$ and  $\mfl\in \pic^{R(l')}(Z_1)$.
Assume that  $\eca^{l', \mfl}$ is nonempty. Then it is smooth of dimension
 $h^1(Z_1,\mfl)  - h^1(Z_1,\calO_{Z_1})+ (l', Z)$ and it is irreducible.
\end{corollary}

Let' recall from \cite{R} the analouge of the theroems about dominance of Abel maps in the relative setup:

\begin{definition}
Fix $l'\in -\calS'$,
$Z\geq E$, $Z_1 \leq Z$ and  $\mfl\in \pic^{R_1(l')}(Z_1)$ as above.
We say that the pair $(l',\mfl ) $ is {\it relative
dominant} if the closure of $ r^{-1}(\mfl)\cap {\rm Im}(c^{l'})$ in  $r^{-1}(\mfl)$
is $r^{-1}(\mfl)$.
\end{definition}

\begin{theorem}\label{th:dominantrel} One has the following facts:

(1) If $(l',\mfl)$ is relative dominant then $ \eca^{l', \mfl}$ is
nonempty and $h^1(Z,\calL)= h^1(Z_1,\mfl)$ for any
generic $\calL\in r^{-1}(\mfl)$.

(2) $(l',\mfl)$ is relative dominant  if and only if for all
 $0<l\leq Z$, $l\in L$ one has
$$\chi(-l')- h^1(Z_1, \mfl) < \chi(-l'+l)-
 h^1((Z-l)_1, \mfl(-R_1(l))).$$, where we denote $(Z-l)_1 = \min(Z-l, Z_1)$.
\end{theorem}

\begin{theorem}\label{rel}
Fix $l'\in -\calS'$,
$Z\geq E$ , $Z_1 \leq Z$ and  $\mfl\in \pic^{R_1(l')}(Z_1)$  as in Theorem \ref{th:dominantrel}. 
Then for any $\calL\in r^{-1}(\mfl)$ one has
\begin{equation}\label{eq:genericLrel}
\begin{array}{ll}h^1(Z,\calL)\geq \chi(-l')-
\min_{0\leq l\leq Z,\ l\in L} \{\,\chi(-l'+l) -
h^1((Z-l)_1, \mfl(-R_1(l)))\, \}, \ \ \mbox{or, equivalently,}\\
h^0(Z,\calL)\geq \max_{0\leq l\leq Z,\, l\in L}
\{\,\chi(Z-l,\calL(-l))+  h^1((Z-l)_1, \mfl(-R_1(l)))\,\}.\end{array}\end{equation}
Furthermore, if $\calL$ is generic in $r^{-1}(\mfl)$
then in both inequalities we have equalities and we have the even stronger statement, that $h^0(Z,\calL) =
\max_{0\leq l\leq Z,\, l\in L, H^0(Z-l, \calL-l)_0 \neq \emptyset}\{\,\chi(Z-l,\calL(-l))+  h^1((Z-l)_1,  \mfl(-R_1(l)))\,\}$.
\end{theorem}

In the following we recall the results from \cite{R} about relatively generic analytic structures:

Let's fix a a topological type, so a resolution graph $\mathcal{T}$ with vertex set $\calv$,

We consider a partition $\calv = \calv_1 \cup  \calv_2$ of the set of vertices $\calv=\calv(\Gamma)$. They define
two (not necessarily connected) subgraphs $\Gamma_1$ and $\Gamma_2$.
We call the intersection of an exceptional divisor from
$ \calv_1 $ with an exceptional divisor from  $ \calv_2 $ a
{\it contact point}. For any $Z\in L=L(\Gamma)$ we write $Z=Z_1+Z_2$,
where $Z_i\in L(\Gamma_i)$ is
supported in $\Gamma_i$ ($i=1,2$).

Furthermore, parallel to the restriction maps
$r_i : \pic(Z)\to \pic(Z_i)$ one also has the (cohomological) restriction operator
  $R_i : L'(\Gamma) \to L_i':=L'(\Gamma_i)$
(defined as $R_i(E^*_v(\Gamma))=E^*_v(\Gamma_i)$ if $v\in \calv_i$, and
$R_i(E^*_v(\Gamma))=0$ otherwise).
For any $l'\in L'(\Gamma)$ and any $\calL\in \pic^{l'}(Z)$ it satisfies
\begin{equation}\label{eq:CHERNREST}
c_1(r_i(\calL))=R_i(c_1(\calL)).
\end{equation}

In the following for the sake of simplicity we will denote $r = r_1$.

Furthermore let's have a fixed analytic type $\tX_1$ for the subtree $\mathcal{T}_1$ (if it is nonconnected, then an analytic type for each component).

Also for each vertex $v_2 \in \calv_2$ which has got a neighbour in $\calv_1$ let's fix a cut $D_{v_2}$ on $\tX_1$ along we glue the exceptional divisor $E_{v_2}$.
If for some vertex $v_2 \in \calv_2$ which has got a neighbour in $\calv_1$ there isn't a fixed cut, then we glue the exceptional divisor $E_{v_2}$ along a generic cut.

We know that the tubular neighbourhood of the exceptional divisors are analitycally taut, so let's plumb the tubular neihgbourhood
of the vertices in $\calv_2$  with the above conditons generically to the fixed resolution $\tX_1$, now we get a singularity $\tX$ and we say that $\tX$ is a relatively generic singularity
corresponding to the analytical structure $\tX_1$ and the cuts $D_{v_2}$ (For more precise definitions see \cite{R}).

We have the following theorem with this setup:

\begin{theorem}\label{rell}
Let's have the setup as above, so two resolution graphs $\mathcal{T}_1 \subset \mathcal{T} $ with vertex sets $\calv_1 \subset \calv$, where $\calv = \calv_1 \cup \calv_2$  and a fixed singularity $\tX_1$ for the resolution graph $\mathcal{T}_1$, and cuts
$D_{v_2}$ along we glue $E_{v_2}$ for all vertices $v_2 \in \calv_2$ which has got a neighbour in $\calv_1$.

Now let's assume that $\tX$ is a relatively generic analytic stucture on $\mathcal{T}$ corresponding to $\tX_1$.

Furthermore let's have an effective cycle $Z$ on $\tX$ and let's have $Z = Z_1 + Z_2$, where $|Z_1| \subset \calv_1$ and $|Z_2| \subset \calv_2$.

Let's have a natural line bundle $\calL$ on $\tX$, such that $c_1(\calL)= l' = - \sum_{v \in \calv'} a_v E_v$, with $a_v > 0, v \in \calv_2 \cap |Z|$, and let's denote $c_1 (\calL | Z) = l'_ m= - \sum_{v \in |Z|} b_v E_v$.

Furthermore let's denote $\mfl = \calL | Z_1$, then we have the following:

1)  We have $H^0(Z,\calL)_0\not=\emptyset$ if and only if $(l',\mfl)$ is relative dominant or equivalently:

\begin{equation}
\chi(-l'_m)- h^1(Z_1, \mfl) < \chi(-l'_m+l)-  h^1((Z-l)_1, \mfl(-R(l))),
\end{equation}
for all $0 \leq l \leq Z$.

2) 
\begin{equation}
h^1(Z, \calL) =  h^1(Z, \calL_{gen}) = \chi(-l'_m) - \min_{0 \leq l \leq Z}(\chi(-l'_m+l)-  h^1((Z-l)_1, \mfl(-R(l)))),                                   
\end{equation}
where $\calL_{gen}$ is a generic line bundle in $ r^{-1}(\mfl) \subset \pic^{l'_m}(Z)$.
\end{theorem}

\section{Dimensions of images of relative Abel maps}

In the following we want to generalise the theorems in \cite{NNAD} about the values $\dim( \im(c^{l'}(Z)))$ for generic singularities to the case of relatively generic singularities,
 in a more general setup:

\begin{theorem}
Let $\mathcal{T}_1 \subset \mathcal{T}$ be $2$ resolution graphs and let $\tX_1$ be a fixed singularity with resolution graph $\mathcal{T}_1$ and assume, that $\tX_1 \subset \tX$ is a singularity with resolution graph $\mathcal{T}$.
Let's have a fixed line bundle $\calL$ on $\tX_1$ and a Chern class $l' \in - S'$, such that $c^1(\calL)  = R(l')$.

For an arbitrary cycle $B$ on $\tX$ let's denote $g(B, \calL, l') = h^1(B, \calL')$, where $\calL'$ is a generic line bundle in $r^{-1}(\calL | B_1) \in \pic^{l'}(B)$.
Notice, that by our theroem on relatively generic line bundles, this is an analytic invariant depending just on $\tX_1$ and $\calL$.

For an arbitrary cycle $Z$  on $\tX$ and let's denote:
\begin{equation}
b_{Z, \calL,  l'} = \max_{Z' \leq Z} \left( \sum_{1 \leq i \leq n} T(Z'_i, \calL,  l') \right),
\end{equation}
where $T(Z'_i, \calL,  l') = g(Z'_i, \calL,  l') + D(Z'_i, \calL,  l')$, $|Z'_1|, \cdots, |Z'_n|$ are the connected components of $|Z'|$ with $\sum_{1 \leq i \leq n}Z'_i = Z'$, and $D(Z'_i, \calL,  l') = 1$ if $(l', \calL | (Z'_i)_1)$ is not relative dominant on $Z'_i$, and $0$ otherwise. 

Assume in the following, that $\eca^{ l', \calL}(Z)$ is nonempty.

1) Let's have a generic line bundle $\mfl \in c^{l'}(\eca^{ l', \calL}(Z))$, then  one has $h^1(Z, \mfl) = h^1(Z_1, \calL) + h^1(Z) - h^1(Z_1) - \dim(c^{l'}(\eca^{ l', \calL}(Z)) )$.

2)  We have the inequality $h^1(Z, \mfl)  \geq  b_{Z, \calL, l'}$.

3)   Assume, that the singularity $\tX$ is relatively generic corresponding to $\tX_1$, then we have the equality $h^1(Z, \mfl) =  b_{Z, \calL,  l'}$.
\end{theorem}

\begin{proof}

To prove part 1) let's have a generic smooth point $D \in \eca^{ l', \calL}(Z) \subset \eca^{l'}(Z)$, where we have $\dim(T_D(\eca^{l'}(Z))) = (l', Z)$ and $\dim(T_D(\eca^{ l', \calL}(Z))) = (l', Z) - h^1(Z_1) + h^1(Z_1, \calL)$.
Let's denote the restriction of the map $c^{l'}$ to the spae $\eca^{ l', \calL}(Z)$ by $c^{l', \calL}$.

Now we want to compute $\dim(\im(T_D(c^{l'}(Z))) )$ in the following, since we have $h^1(Z,  \mfl) = h^1(Z) - \dim(\im(T_D(c^{l'}(Z))) )$ from \cite{NNA1}.

Let's denote the projection map $\pic^{l'}(Z) \to \pic^{R(l')}(Z_1)$ by $\pi$, and notice, that we have $ \dim(\im(T_D( \pi \circ c^{l'}(Z)  )) ) = h^1(Z_1) - h^1(Z_1, \calL)$.

It means, that $\dim( \pi(  \im(T_D(c^{l'}(Z))) )) =  h^1(Z_1) - h^1(Z_1, \calL)$.

Now we want to compute $ \dim(\im(T_D(c^{l'}(Z))) \cap \pi^{-1} (\calL))$.

We know, that $\dim(\eca^{l'}(Z)) - \dim(\eca^{l', \calL}(Z)) = h^1(Z_1) - h^1(Z_1, \calL)$ and $D$ is a smooth point of $\eca^{ l', \calL}(Z)$ so we get that:

\begin{equation*}
\dim( \pi(  \im(T_D(c^{l'}(Z))) )) = \dim(T_D(\eca^{l'}(Z))) - \dim(T_D(\eca^{ l', \calL}(Z))).
\end{equation*}

We know, that the tangent map of $c^{l'}(Z)$ brings $T_D(\eca^{ l', \calL}(Z))$ to $\pi^{-1} (\calL)$, so it means, that the normal space at $D \in \eca^{l'}(Z)$
to the submanifold $\eca^{ l', \calL}(Z)$ goes bijectively to $ \pi(  \im(T_D(c^{l'}(Z))) )$ by the tangent map of $c^{l'}(Z)$ composed by $\pi$.

It means, that we have:

\begin{equation*}
\im(T_D(c^{l'}) ) \cap \pi^{-1} (\calL) = \im(T_D(c^{l', \calL})).
\end{equation*}

On the other hand we know, that $\dim( \im(T_D(c^{l', \calL})) ) = \dim(c^{l'}(\eca^{ l', \calL}(Z)) )$, which means that:

\begin{equation*}
\dim( \im(T_D(c^{l'})) ) = \dim( \pi(  \im(T_D(c^{l'}(Z))) )) + \dim( \im(T_D(c^{l'}) ) \cap \pi^{-1} (\calL) ) .
\end{equation*}

\begin{equation*}
\dim(\im(T_D(c^{l'})) ) =  \dim(c^{l'}(\eca^{ l', \calL}(Z)) ) +  h^1(Z_1) - h^1(Z_1, \calL).
\end{equation*}

\begin{equation*}
h^1(Z, \mfl) =  h^1(Z_1, \calL) + h^1(Z) - h^1(Z_1) - \dim(c^{l'}(\eca^{l', \calL}(Z)) ).
\end{equation*}

For part 2) let's have an arbitrary cycle $Z' \leq Z$ and let the connected components of $|Z'|$ be $|Z'_1|, \cdots, |Z'_n|$, now we have to prove, that:

\begin{equation*}
h^1(Z, \mfl) \geq \sum_{1 \leq i \leq n} T(Z'_i, \calL,  l').
\end{equation*}

Notice first, that $h^1(Z,  \mfl) \geq h^1(Z', \mfl) = \sum_{1 \leq i \leq n} h^1(Z'_i, \mfl  | Z'_i)$, so we only have to prove, that $h^1(Z'_i, \mfl) \geq  g(Z'_i, \calL,  l') + D(Z'_i, \calL,  l')$.

The statement $h^1(Z'_i, \mfl) \geq  g(Z'_i, \calL,  l')$ is always true by semicontinuity and  if $ D(Z'_i, \calL,  l') = 1$, then if $\calL'$ is a generic line bundle in $r^{-1}(\calL_1) \subset \pic^{l'}(Z'_i)$, where $\calL_1 = \calL | (Z'_i)_1 $, then $H^0(Z'_i, \calL')_0 = \emptyset$ and $H^0(Z'_i, \mfl | Z'_i)_0 \neq \emptyset$ which really means the strong inequality $h^1(Z'_i, \mfl  | Z'_i) \geq  g(Z'_i, \calL,  l') + 1$.

This means, that $h^1(Z'_i, \mfl  | Z'_i) \geq  g(Z'_i, \calL,  l') + D(Z'_i, \calL,  l')$ is always true and part 2) is proved.

For part 3) notice first that if the singularity $\tX$ is relatively generic corresponding to $\tX_1$, then the nonemptyness of  $\eca^{ l', \calL}(Z)$ is equivalent to $H^0(Z_1, \calL)_0 \neq \emptyset$.

Now we prove part 3) first in the case, when every differential form in $\frac{H^0(\tX, K + Z)}{H^0(\tX, K)}$ has got a pole on $E_u$ of order at most $1$ if $u \in \calv \setminus \calv_1$
is an arbitrary vertex.

We want to  compute $\dim(c^{l'}(\eca^{ l', \calL}(Z)))$ in the following.

Now let's denote by $Z^*$ the minimal cycle such that $H^1(Z^*) = H^1(Z)$ and $0 \leq Z^* \leq Z$.

It's obvious in this case that $Z^*_u = 0$ or $Z^*_u = 1$ for every vertex $u \in \calv \setminus \calv_1$.

Let's denote $Z^*_1 = \min(Z^*, Z_1)$ and the line bundle $\calL^* = \calL | Z^*_1$.
Notice, that we have immediately, that $\dim(c^{l'}(\eca^{l', \calL}(Z))) = \dim(c^{l'}(\eca^{l', \calL^*}(Z^*)))$.

Now assume, that there is a cycle $Z' \leq Z$, with connected components  $|Z'_1|, \cdots, |Z'_n|$, then let's denote $Z'' = \min(Z^*, Z')$ and $Z''_j = \min(Z^*, Z'_j)$, we know, that $Z'' = \sum_{1 \leq j \leq n} Z''_j$.

However the cycles $Z''_j$ may not be connected, so let's denote the connected components of $Z''_j$ by $Z''_{j, k}, 1 \leq k \leq t_j$.

Notice, that we have $g(Z'_j, \calL,  l') = \sum_{1 \leq k \leq t_j} g(Z''_{j, k}, \calL^*,  l') $ and $\sum_{1 \leq k \leq t_i}D(Z''_{j, k} , \calL^*,  l') \geq D(Z'_{j} , \calL,  l')$, which means, that $b_{Z, \calL,  l'} \leq  b_{Z^*, \calL^*,  l'}$.

On the other hand $b_{Z, \calL,  l'} \geq  b_{Z^*, \calL^*,  l'}$ is trivial by the definition, so we get, that $b_{Z, \calL,  l'} =  b_{Z^*, \calL^*,  l'}$, this means, that we are enough to prove the statement in the case, when $Z = Z^*$.

If there is a vertex $u \in \calv \setminus \calv_1$, suh that $Z_u = 0$ and the support $|Z|$ is disconnected, then we can prove the statement independently for
the connected components of $Z$.

So assume in the following, that $Z_u = 1$ for every vertex $u \in \calv \setminus \calv_1$, and let's denote $I = \calv \setminus \calv_1$.

Notice first, that if $u \in I$, then every differential form in $\frac{H^0(\tX, K + Z)}{H^0(\tX, K )}$ has got a pole on $E_{u}$ of order at most $1$, 
which means, that $\dim(\im(c^{-E_{u}^*}(Z)))= 1$ and $\dim(A(\im(c^{-E_{u}^*}(Z)))) = e_{Z}(u)$, where $A(\im(c^{-E_{u}^*}(Z)))$ is the affine clousure of
$\im(c^{-E_{u}^*}(Z))$.

We know from part 2), that $e_{Z}(I) - \dim(c^{l'}(\eca^{ l', \calL}(Z))) \geq b_{Z, \calL,  l'} -  h^1(Z_1, \calL)$, so in the following we are enough to prove the other direction
$e_{Z}(I)  - \dim(c^{l'}(\eca^{ l', \calL}(Z))) \leq b_{Z, \calL,  l'} - h^1(Z_1, \calL)$.

Notice, that if $U \subset r^{-1}(\calL)$ is an irreducible variety, than $\dim(U) =  e_{Z}(I) - t$, where $t$ is the minimal nonnegative integer such that there exists a Chern class
$l'_2 \in - S'$, which is supported on $I$, $(l'_2, Z) = t$ and $\dim(U \oplus \im(c^{l'_2}(Z))) =  e_{Z}(I)$.

It means, that $\dim(c^{l'}(\eca^{ l', \calL}(Z))) =  e_{Z}(I) - t$, where $t$ is the minimal nonnegative integer such that there exists a Chern class
$l'_2 \in - S'$, which is supported on $I$, $(l'_2, Z) = t$  and  $(l' + l'_2, \calL)$ is relative dominant on $Z$.

We prove in the following, that there exists  a Chern class $l'_2 \in - S'$,  which is supported on $I$, $(l'_2, Z) \leq T(Z, \calL,  l') - h^1(Z_1, \calL) $ and
 $(l' + l'_2, \calL)$ is relative dominant on $Z$.
This proves our claim since $b_{Z, \calL,  l'} \geq T(Z, \calL,  l')$.

Now if $(l' , \calL)$ is relative dominant on $Z$, then we can choose $l'_2 = 0$, which means, that $(l'_2, Z) = 0$, on the other hand in this case $T(Z, \calL,  l') - h^1(Z_1, \calL) = 0$ by definition, which proves our claim.

So assume, that $(l' , \calL)$ is not relatively dominant on $Z$, then we get that $T(Z, \calL,  l') = g(Z, \calL,  l') + 1 = \chi(-l') - \min_{0 \leq l \leq Z}( \chi(-l' +l) - h^1((Z-l)_1, \calL( - R(l))  )   ) + 1$.

So we have to find an appropriate Chern class $l'_2 \in - S'$, such that $(l'_2, Z) \leq (\chi(-l') - h^1(Z_1, \calL)) - \min_{0 \leq l \leq Z}( \chi(-l' +l) - h^1((Z-l)_1, \calL( - R(l))  )   ) + 1$.

Let's prove now the following lemma:

\begin{lemma}
Let's have two resolution graphs $\mathcal{T}_1 \subset \mathcal{T}$ and two singularities $\tX_1 \subset \tX$ corresponding to them, a cycle $Z$ on $\tX$, a line bundle $\calL$ on the
subsingularity $\tX_1$ and a Chern class $l' \in -S'$ such that $c^1(\calL) = R(l')$.

Let's look at the cycles $0 \leq l_m \leq Z$, such that $\min_{0 \leq l \leq Z}( \chi(-l' +l) - h^1((Z-l)_1, \calL( - R(l))  )  ) = \chi(-l' +l_m) - h^1((Z-l_m)_1, \calL( - R(l_m))) $ and let's denote the set of them by $L_m$, then $L_m$ has got a maximal element.

Let's look at the cycles $0 \leq l_{m, 2} \leq Z_2$, for which there exists a cycle $0 \leq l_m \leq Z$, such that $(l_m)_2 = l_{m, 2}$ and $\min_{0 \leq l \leq Z}( \chi(-l' +l) - h^1((Z-l)_1, \calL( - R(l))  )  ) = \chi(-l' +l_m) - h^1((Z-l_m)_1, \calL( - R(l_m))) $ and let's denote the set of them by $L_{m, 2}$, then $L_{m, 2}$ has got a minimal element.

\end{lemma}
\begin{proof}

Let's look at a generic line bundle in $r^{-1}(\calL) \subset \pic^{l'}(Z)$ and let's denote it by $\mfl$, notice that if $0 \leq l \leq Z$ is an arbitrary cycle, then $\mfl(-l)| Z-l$
is a generic line bundle in $r^{-1}(\calL(-R(l))) \subset \pic^{l'}(Z-l)$, so by the main theorem about relatively generic line bundles from \cite{R} we get:

\begin{equation}
h^0(Z-l, \mfl(-l)| Z-l) =\chi(-l' + l)  + \chi(Z-l,  \mfl(-l))  - \min_{0 \leq l _2 \leq Z-l}( \chi(-l' +l+ l_2) - h^1((Z-l - l_2)_1, \calL( - R(l + l_2))  )  ) .
\end{equation}

We know, that there is a unique cycle $0 \leq l_M \leq Z$ such that $h^0(Z-l_M, \mfl(-l_M)| Z-l_M) = h^0(Z, \mfl)$ and $H^0(Z-l_M, \mfl(-l_M)| Z-l_M)_0 \neq 0$.

From the previous equation we get that $ l_M$  is the maximal cycle in the set $L_m$, so we have proved that $L_m$ has got a maximal element.

In the following we prove the existence of a minimal element in $L_{m, 2}$:

Let's have an arbitrary cycle $0 \leq Z' \leq Z$ such that $(Z')_1 = Z_1$, then we know, that $\mfl | Z'$ is a  generic line bundle in $r^{-1}(\calL | Z'_1) \subset \pic^{l'}(Z')$, which means, that:

\begin{equation}
h^1(Z', \mfl | Z') =\chi(-l')   - \min_{0 \leq l  \leq Z'}( \chi(-l' +l) - h^1((Z'-l)_1, \calL( - R(l))) ).
\end{equation}

It means, that if $(Z')_1 = Z_1$, then we have $h^1(Z', \mfl | Z') = h^1(Z, \mfl )$ if there exists an element $ l_{m, 2} \in L_{m, 2}$, such that $ l_{m, 2} \leq Z'$.

We know, that if $0\leq Z', Z'' \leq Z$ are two cycles, such that $h^1(Z', \mfl | Z') = h^1(Z, \mfl )$ and $h^1(Z'', \mfl | Z'') = h^1(Z, \mfl )$, then we have $h^1(\min(Z'', Z'), \mfl | \min(Z'', Z')) = h^1(Z, \mfl )$.

Now assume, that the set $L_{m, 2}$ hasn't got a minimal element, it means, that there are two different elements $l_{m, a}, l_{m, b} \in L_{m, 2}$, such that if $l < l_{m, a}$, 
then $l \notin L_{m, 2}$ and if $l < l_{m, b}$, then $l \notin L_{m, 2}$.

Now we get, that $h^1(Z_1 + l_{m, a}, \mfl | Z_1 + l_{m, a}) = h^1(Z, \mfl )$ and $h^1(Z_1 + l_{m, b}, \mfl | Z_1 + l_{m, b}) = h^1(Z, \mfl )$, but we have
$h^1(Z_1 + \min( l_{m, a},l_{m, b}) , \mfl | Z_1 +\min( l_{m, a},l_{m, b})) < h^1(Z, \mfl )$, which is a contradiction, this proves the second statement of the lemma.
\end{proof}

Now we determine $l'_2$ recursively, we will determine a set of Chern classes $l''_t$ recursively, such that $l''_0 = 0$ and $l''_t = l''_{t-1} - E_u^*$ for some vertex $u \in I$.

Now, suppose we have determined $l''_{t-1}$ and $(l' + l''_{t-1}, \calL)$ is relative dominant on $Z$, then we stop the algorithm and set $l'_2 = l''_{t-1}$.

Suppose on the other hand, that $(l' + l''_{t-1}, \calL)$ is not relative dominant on $Z$, then we have:

\begin{equation}
\min_{0 \leq l \leq Z}( \chi(-l'  - l'_{t-1} +l) - h^1((Z-l)_1, \calL( - R(l))  )   ) < \chi(-l'  - l'_{t-1}) - h^1(Z_1, \calL).
\end{equation}

Let's look at the cycles $0 \leq l_{m, 2} \leq Z_2$, such that there exists a cycle $0 \leq l_m \leq Z$, for which $(l_m)_2 = l_{m, 2}$ and $\min_{0 \leq l \leq Z}( \chi(-l'  - l'_{t-1} +l) - h^1((Z-l)_1, \calL( - R(l))  )  ) = \chi(-l'  - l'_{t-1}+l_m) - h^1((Z-l_m)_1, \calL( - R(l_m))) $ and let's denote the set of them by $L_{m, 2}$, then by our previous lemma $L_{m, 2}$ has got a minimal element, let's denote it by $l_{2, t-1}$.

Let's choose an arbitrary vertex $u \in |l_{2, t-1}|$ and set $l''_t = l''_{t-1} - E_u^*$.

Let's denote in the following $d_t = \chi(-l'  - l'_{t}) - h^1(Z_1, \calL) - \min_{0 < l \leq Z}( \chi(-l'  - l'_{t} +l) - h^1((Z-l)_1, \calL( - R(l))  )   ) $, where we have $d_0 =  T(Z, \calL,  l') - h^1(Z_1, \calL) $ and $d_t < 0$ if and only if $(l' + l''_{t}, \calL)$ is relatively dominant on $Z$, so in particular $d_0 \geq 0$.

We claim in the following, that $d_t < d_{t-1}$ in case both of them are defined.

Indeed we have to prove that if $0 \leq l \leq Z$ is an arbitrary cycle then $\chi(-l'  - l'_{t}) - h^1(Z_1, \calL) - ( \chi(-l'  - l'_{t} +l) - h^1((Z-l)_1, \calL( - R(l))  )   ) < d_{t-1}$.

If $l \ngeq l_{2, t-1}$, then we have $ \chi(-l'  - l'_{t-1}) - h^1(Z_1, \calL) - ( \chi(-l'  - l'_{t-1} +l) - h^1((Z-l)_1, \calL( - R(l))  )   ) < d_{t-1}$ and we obviously have

\begin{equation}
(-l'  - l'_{t-1}, l) - h^1(Z_1, \calL) + h^1((Z-l)_1, \calL(- R(l))  )  \geq (-l'  - l'_{t}, l) - h^1(Z_1, \calL) + h^1((Z-l)_1, \calL(- R(l))  )  ,
\end{equation}
which indeed implies that $\chi(-l'  - l'_{t}) - h^1(Z_1, \calL) - ( \chi(-l'  - l'_{t} +l) - h^1((Z-l)_1, \calL( - R(l))  )   ) < d_{t-1}$.

Assume on the other hand that $l \geq l_{2, t-1}$, then we have $ \chi(-l'  - l'_{t-1}) - h^1(Z_1, \calL) - ( \chi(-l'  - l'_{t-1} +l) - h^1((Z-l)_1, \calL( - R(l))  )   ) \leq d_{t-1}$
and we have
\begin{equation}
(-l'  - l'_{t-1}, l) - h^1(Z_1, \calL) + h^1((Z-l)_1, \calL(- R(l))  )  > (-l'  - l'_{t}, l) - h^1(Z_1, \calL) + h^1((Z-l)_1, \calL(- R(l))  )  ,
\end{equation}
because $(-E_u^*, l) \geq (-E_u^*, l_{2, t-1}) \geq 1$, which proves our claim in this case too.

Now assume, that the algorithm stops at $t = k$, this means, that $k \leq d_0 + 1 = T(Z, \calL,  l') - h^1(Z_1, \calL) + 1$.

Now if we denote $l''_k = l'_2$, then $(l' + l'_{2}, \calL)$ is relatively dominant on $Z$ and $(l'_2, Z) = k  \leq T(Z, \calL,  l') - h^1(Z_1, \calL) + 1$.

This proves part 3) in the case, when every differential form in $\frac{H^0(\tX, K + Z)}{H^0(\tX, K)}$ has got a pole on $E_u$ of order at most $1$ if $u \in \calv \setminus \calv_1$
is an arbitrary vertex.

In the following we prove part 3) in the special case when $(l', E_u) = 0$ for all vertices $u \in \calv_2$ by induction on $h^1(Z)$.

If $h^1(Z) = 0$ then the statement is trivial, so assume in the following, that $h^1(Z) = t > 0$ and the statement is true for $h^1(Z) < t$.

As before we can assume, that $H^0(\tX, K+ Z)_0 \neq 0$.

The subgraph induced by the vertex set $\calv_2$ may not be connected, let the connected components be $\calv_{2, i}, 1 \leq i \leq r$.

Notice first that we want to compute $h^1(Z, \mfl)$, where $\mfl \in c^{l'}(\eca^{l', \calL}(Z))$ is generic.

This means, that we glue the tubular neihborhoods of the exceptional dvisors $E_u, u \in \calv_2$ generically to the fixed singularity $\tX_1$ and then we pick a generic divisor
$D \in (c^{l'})^{-1}(\calL) \in \eca^{R_1(l')}(Z_1)$ and compute $h^1(Z, D)$.

However we can reverse this process by fixing a generic divisor $D$ and then glue the tubular neighborhoods generically.
If we compute $h^1(Z, D)$ this way, then we get the same result as before obviously.
So in the following let's fix a generic divisor $D \in (c^{l'})^{-1}(\calL) \in \eca^{R_1(l')}(Z_1)$.

In the following for every component $\calv_{2, i},  1 \leq i \leq r$ we choose a vertex $u_i \in \calv_{2, i}$.

We know, that the differential forms in $\frac{H^0(\tX, K + Z)}{H^0(\tX, K)}$ have got a pole on $E_{u_i}$ of order at most $Z_{u_i} \geq 1$ and there exist differential forms in it
which have got pole on the exceptional divisor $E_{u_i}$ of order $Z_{u_i}$, let's denote $k_i = Z_{u_i}$.

Now let's blow up the singularity $\tX$ along the divisors $E_{u_i}$ sequentially $k_i-1$ times at generic points and let the new singularity be $\tX_{new}$ with resolution graph $\mathcal{T}_{new}$ and let the new exceptional divisors be $E_{w_{i, 1}}, \cdots E_{w_{i, k_i -1}}$, where $w_{i, k_i -1} = u_i$ if $k_i = 1$.

Furthermore let's have the cycle $Z_{new} = \pi^*(Z)$ on $\tX_b$ and the Chern class $\pi^*(l') \in -S'_{new}$.

Notice that every differential form in $\frac{H^0(\tX, K + Z)}{H^0(\tX, K)}$ have got a pole on the divisors $E_{w_{i, k_i - 1}}$ of order at most $1$ and there exists a differential form
which has got a pole of order $1$, so it means, that $e_{Z_{new}, w_{i, k_i -1}} > 0$ for each $1 \leq i \leq r$.

Let's denote the subsingularity of $\tX_b$ corresponding to the vertex set $\calv_{new} \setminus \cup_{1 \leq i \leq r} w_{i, k_i -1}$ by $\tX_u$, the restriction of the cycle $Z_{new}$ by $Z_u$ and the restriction of the Chern class $\pi^*(l')$ still by $l'$.
We know, that $\tX_{new}$ is a relatively generic singularity corresponding to $\tX_u$ and $\tX_{u}$ is a relatively generic singularity corresponding to $\tX_1$.
Now we have $h^1(Z_u) < h^1(Z_{new}) = h^1(Z)$.

We claim first that if we have a generic line bundle $\mfl_{new} \in c^{\pi^*(l')}(\eca^{\pi^*(l'), \calL}(Z_{new}))$, then $h^1(Z, \mfl) = h^1(Z_{new}, \mfl_{new})$ and
$b_{Z, \calL,  l'} = b_{Z_{new}, \calL,  \pi^*(l')}$.

The first statement is trivial because if $\mfl \in c^{l'}(\eca^{l', \calL}(Z))$ is a generic line bundle, then $\pi^*(\mfl) \in c^{\pi^*(l')}(\eca^{\pi^*(l'), \calL}(Z_{new}))$ is a generic
line bundle, and we have $h^1(Z, \mfl) = h^1(Z_{new}, \pi^*(\mfl))$.

We show in the following, that $b_{Z, \calL,  l'} = b_{Z_{new}, \calL,  \pi^*(l')}$, we prove this statement in the following lemma:

\begin{lemma}
Let's blow up vertices in $I$ arbitrary times, let the new singularity be $\tX_{new}$, furthermore  let's have the cycle $Z_{new} = \pi^*(Z)$ on $\tX_{new}$ and the Chern class $\pi^*(l') \in -S'_{new}$, then we have $b_{Z, \calL,  l'} = b_{Z_{new}, \calL,  \pi^*(l')}$.
\end{lemma}
\begin{proof}

By induction we are enough to prove the lemma in the case we blow up one vertex, let's denote it by $u$ and let's denote the new vertex by $u_{new}$.

Assume first that $b_{Z, \calL,  l'} = \sum_{1 \leq i \leq n} T(Z'_i, \calL,  l')$, where $Z' \leq Z$ and $Z'_1, \cdots, Z'_n$ are the connected components of $Z'$.

Let's look at the cycles $\pi^*(Z')$ and $\pi^*(Z'_1), \cdots, \pi^*(Z'_n)$ on $\tX_{new}$, we know, that $\pi^*(Z'_1), \cdots, \pi^*(Z'_n)$ are the connected components of
$\pi^*(Z')$ and $\pi^*(Z') \leq \pi^*(Z)$.

On the other hand for every $1 \leq i \leq n$ we have $T(Z'_i, \calL,  l') = T(\pi^*(Z'_i), \calL,  \pi^*(l'))$, so we get that:

\begin{equation}
b_{Z, \calL,  l'} = \sum_{1 \leq i \leq n} T(\pi^*(Z'_i), \calL,  \pi^*(l')) \leq b_{Z_{new}, \calL,  \pi^*(l')}.
\end{equation} 

It means, that we should prove that $b_{Z, \calL,  l'} \geq b_{Z_{new}, \calL,  \pi^*(l')}$.

Assume, that $b_{Z_{new}, \calL,  \pi^*(l')} = \sum_{1 \leq i \leq n} T(Z'_i, \calL,  \pi^*(l'))$, where $Z' \leq Z_{new}$ and $Z'_1, \cdots, Z'_n$ are the connected components of $Z'$.

If there is a component $Z'_j$ whose support is $u_{new}$, then $T(Z'_j, \calL,  \pi^*(l')) = 0$, so we can leave that component.

If there isn't a component $Z'_j$ for which $u \in |Z'_j|$, then we can look at the cycles $Z'_i$ on the original resolution $\tX$ and we get that $T(Z'_i, \calL,  \pi^*(l')) = T(Z'_i, \calL,  l')$,
which means that $b_{Z_{new}, \calL,  \pi^*(l')} = \sum_{1 \leq i \leq n} T(Z'_i, \calL,  l') \leq b_{Z, \calL,  l'}$.

So assume in the following, that $u \in Z'_1$ and $u, u_{new} \notin Z'_i$ if $i \geq 2$.

Now again if $i \geq 2$ and we look at the cycles $Z'_i$ on the original resolution $\tX$, then we have  $T(Z'_i, \calL,  \pi^*(l')) = T(Z'_i, \calL,  l')$.

On the other hand let's have $Z'_1 = A + r \cdot E_u + h \cdot E_{u_{new}}$ and let's look at the cycle $Z_1 = A + r \cdot E_u$ on $\tX$,
then it's obvious to see, that $T(Z'_1, \calL,  \pi^*(l')) \leq T(Z_1, \calL,  l')$, and it yields:

\begin{equation}
b_{Z_{new}, \calL,  \pi^*(l')} \leq \sum_{2 \leq i \leq n} T(Z'_i, \calL,  l') + T(Z_1, \calL,  l') \leq b_{Z, \calL,  l'}.
\end{equation}

It means, that we proved that indeed $b_{Z, \calL,  l'} = b_{Z_{new}, \calL,  \pi^*(l')}$.

\end{proof}

It means, that we should prove $ h^1(Z_{new}, \mfl_{new}) =  b_{Z_{new}, \calL,  \pi^*(l')}$.

Let's denote the restriction of the line bundle $\mfl_{new}$ to $Z_u$ by $\mfl_{u}$, where we know, that $\mfl_{u} \in c^{l'}(\eca^{l', \calL}(Z_{u}))$ is a generic line bundle given by the divisor $D$.

Similarly we can consider $\mfl_{new}$ as a generic line bundle in $ c^{ \pi^*(l')}(\eca^{ \pi^*(l'), \mfl_{u}}(Z_{new}))$ also given by the divisor $D$.

Notice that we have proved part 3) already in the special case when the differential forms have got pole of order at most $1$ on the exceptional divisor $E_v, v \in \calv_2$,
so we can use this to the situation of the cycle $Z_{new}$ and the Chern class $\pi^*(l')$, where we have the line bundle $\mfl_{u}$ on the cycle $Z_u$, because the differential forms in $\frac{H^0(\tX, K + Z)}{H^0(\tX, K)}$ have got a pole on the divisor $E_{w_{i, k_i - 1}}$ of order at most $1$.

It means, that $h^1(Z_{new}, \mfl_{new}) =  b_{Z_{new}, \mfl_{u},  \pi^*(l')}$, so there exists a cycle $Z''$ with components $Z''_1, \cdots, Z''_m$, such that:

\begin{equation}
h^1(Z_{new}, \mfl_{new}) = \sum_{1 \leq i \leq m} T(Z''_i, \mfl_{u}, \pi^*(l')).
\end{equation}

Assume, that $m$ is minimal with that porperty.

Assume first, that there is a vertex $u \in  \calv_{new} \setminus \calv_1$, such that $u \notin |Z''|$.

Let's denote the restriction of the cycle $Z_{new}$ to the vertex set $\calv_{new} \setminus u$ by $Z_s$, then we have $h^1(Z_s) < h^1(Z_{new}) = h^1(Z)$.

Notice, that obviously $h^1(Z_s, D) \geq h^1(Z'', D) \geq  \sum_{1 \leq i  \leq m} T(Z''_i, \mfl_{u}, \pi^*(l'))$, so it follows, that $h^1(Z_s, D) = h^1(Z_{new}, D)$.

Since the restriction of the cycle $Z_s$ to $\calv_1$ is $Z_1$ we know, that the divisor $D$ on $Z_s$ gives a generic line bundle in $c^{\pi^*(l')}(\eca^{\pi^*(l'), \calL}(Z_{s}))$.

We have $h^1(Z_s) <  h^1(Z)$, so by the induction hypothesis we get, that there is a cycle $Z^*$ with components $Z^*_1, \cdots, Z^*_q$, such that:

\begin{equation}
h^1(Z_{s}, D) = \sum_{1 \leq j \leq q} T(Z^*_j, \calL , \pi^*(l')).
\end{equation}

It means, that $h^1(Z_{new}, D) = \sum_{1 \leq j \leq q} T(Z^*_j, \calL , \pi^*(l')) \leq b_{Z_{new}, \calL,  \pi^*(l')}$ which proves the statement in this case.

So we got that $u \in |Z''|$  for every $u \in \calv_{new} \setminus \calv_1$.

On the other hand we claim, that if  $1 \leq j \leq m$, then there is a vertex $u \in \calv_{new} \setminus \calv_1$, such that $u \in |Z''_j|$.

Assume to the contrary, that $ |Z''_j| \subset \calv_1$ and let's have another component $Z''_k$, which is a neighbour of $Z''_j$ in the sense that the vertices on the path between them
does not belong to $|Z''|$.

Let $Z''_{k, j}$ be the smallest connected cycle which is bigger, then $Z''_k + Z''_j$, then we get easily that $T(Z''_j, \calL , \pi^*(l')) + T(Z''_k, \calL , \pi^*(l')) \leq T(Z''_{k, j}, \calL , \pi^*(l'))$, which means, that $h^1(Z'', D) =  \sum_{1 \leq i \leq m, i \neq k, j} T(Z''_i, \mfl_{u}, \pi^*(l')) + T(Z''_{k, j}, \calL , \pi^*(l'))$, which contradicts the minimality of $m$.

This means, that for all $1 \leq j \leq m$ the support $|Z''_j|$ consists of a subset of $\calv_1$ and a few connected components of $ \calv_{new} \setminus \calv_1$ and $|Z''_j| \cap \calv_2 \neq \emptyset$.

For $1 \leq j \leq m$  let's denote by $B_j \subset (1, \cdots, r)$ the subset of connected components of $ \calv_{new} \setminus \calv_1$ which is contained in $|Z''_j|$, we got that $B_1, \cdots, B_m$
gives a partition of $ (1, \cdots, r)$.

This means also that for $1 \leq j \leq m$  and $i \in B_j$ one has $w_{i, k_i - 1} \in |Z''_j|$.

We can also assume, that for every $1 \leq j \leq m$ we have $D(Z''_j, \mfl_{u}, \pi^*(l')) = 1$.
Indeed assume, that $D(Z''_j,  \mfl_{u}, \pi^*(l')) = 0$, then we have $T(Z''_j, \mfl_{u}, \pi^*(l')) = T((Z''_j)_u, \mfl_{u}, \pi^*(l'))$ and then we are back to the case when there is a vertex
$u \in  \calv_{new} \setminus \calv_1$ such that $u \notin |Z''|$.
So in the following we assume that $D(Z''_j, \mfl_{u}, \pi^*(l')) = 1$ for every $1 \leq j \leq m$.

For $1 \leq j \leq m$ let's denote by $\mfl_j$ the generic line bundle in $r_j^{-1}(\calO_{(Z''_j)_u}(D)  ) \subset \pic^{R_j(\pi^*(l'))}(Z''_j)$, then by definition we have $T(Z''_j, \mfl_{u}, \pi^*(l')) = h^1(Z''_j, \mfl_j) + 1$  and we know, that $h^1(Z''_j, D) = T(Z''_j, \mfl_{u}, \pi^*(l')) = h^1(Z''_j, \mfl_j) + 1$ for every $1 \leq j \leq m$.

From $D(Z''_j, \mfl_{u}, \pi^*(l')) = 1$ we know, that $H^0(Z''_j, \mfl_j)_0 = \emptyset$, which means, that there is a smallest cycle $l_j$, such that $H^0(Z''_j - l_j, \mfl_j - l_j)_0 \neq \emptyset$.

Assume in the following first, that there exists an index $1 \leq j \leq m$ and $i \in B_j$, such that $E_{w_{i, k_i - 1}} \nleq l_j$.

Notice that if $w'$ is a neighbour of $w_{i, k_i - 1}$, then we also have $E_{w'} \nleq l_j$ because $( \pi^*(l'), E_{w_{i, k_i - 1}}) = 0$ and $(Z''_j - l_j, \pi^*(l') - l_j) \geq 0$.

Notice, that $\calO_{Z''_j - l_j}(\mfl_j - l_j)$ is a relatively generic line bundle in $r_j^{-1}(\calO_{(Z''_j - l_j)_u}(D-l_j)  ) \in \pic^{\pi^*(l') - l_j}(Z''_j - l_j)$ and we have $H^0(Z''_j - l_j, \mfl_j - l_j)_0 \neq \emptyset$, which means by our main theorem about relatively generic line bundles from \cite{R}, that $h^1((Z''_j - l_j)_u, \mfl_j - l_j) = h^1(Z''_j- l_j, \mfl_j - l_j)$

This also means, that $h^1(Z''_j- l_j, \mfl_j - l_j) = h^1(Z''_j- l_j - E_{w_{i, k_i - 1}}, \mfl_j - l_j)$.

We have $h^0(Z''_j - l_j, \mfl_j - l_j) = h^0(Z''_j , \mfl_j)$, which means, that $ h^1(Z''_j, \mfl_j) = h^1(Z''_j- l_j, \mfl_j - l_j) + \chi(\pi^*(l')) - \chi(\pi^*(l') + l_j)$.

On the other hand we have $ h^1(Z''_j - E_{w_{i, k_i - 1}}, \mfl_j) \geq h^1(Z''_j- l_j - E_{w_{i, k_i - 1}}, \mfl_j - l_j) + \chi(\pi^*(l')) - \chi(\pi^*(l') + l_j)$, which means, that 
$ h^1(Z''_j - E_{w_{i, k_i - 1}}, \mfl_j) \geq h^1(Z''_j , \mfl_j)$, which means of course, that $ h^1(Z''_j- E_{w_{i, k_i - 1}}, \mfl_j) = h^1(Z''_i , \mfl_j)$.

Notice, that $\mfl_j | Z''_j- E_{w_{i, k_i - 1}}$ is a generic line bundle in $r_j^{-1}(\calO_{(Z''_j - E_{w_{i, k_i - 1}})_u}(D)  ) \in \pic^{\pi^*(l')}(Z''_j - E_{w_{i, k_i - 1}})$.

We know, that $H^0(Z''_j - E_{w_{i, k_i - 1}}, D)_0 \neq \emptyset$ so we have $h^1(Z''_j - E_{w_{i, k_i - 1}}, D) > h^1(Z''_j - E_{w_{i, k_i - 1}} - l_j, D- l_j) + \chi(\pi^*(l')) - \chi(\pi^*(l') + l_j)$.

On the other hand we have $h^1(Z''_j - E_{w_{i, k_i - 1}} - l_j, D- l_j) \geq  h^1(Z''_i - l_j - E_{w_{i, k_i - 1}}, \mfl_j - l_j)$, so we get that $h^1(Z''_j - E_{w_{i, k_i - 1}}, D) \geq h^1(Z''_j , \mfl_j) + 1 = h^1(Z''_j, D)$.

Now we have $h^1(Z, D) = \sum_{1 \leq k \leq m} h^1(Z''_k, D) = \sum_{1 \leq k \leq m, k \neq j} h^1(Z''_k , D) + h^1(Z''_j - E_{w_{i, k_i - 1}}, D) $, it means, that we get $ h^1(Z, D) = h^1(Z_{new} - E_{w_{i, k_i - 1}} , D)$.

We know, that $(Z_{new} - E_{w_{i, k_i - 1}})_1 = Z_1$ so $D$ is a generic divisor of the line bundle $\calO_{(Z - E_{w_{i, k_i - 1}})_1}(D) = \calO_{Z_1}(D)$
and $h^1(Z_{new} - E_{w_{i, k_i - 1}}) < h^1(Z)$ so by the induction hypothesis we know, that there is a cycle $Z^* \leq Z_{new} - E_{w_{i, k_i - 1}}$ with components $Z^*_1, \cdots Z^*_q$, such that:

\begin{equation}
h^1(Z_{new} - E_{w_{i, k_i - 1}} , D) = \sum_{1 \leq k \leq q} T(Z^*_k, \calL, \pi^*(l')).
\end{equation}

However we get that $ h^1(Z_{new}, D) = \sum_{1 \leq k \leq q} T(Z^*_k, \calL, \pi^*(l')) \leq b_{Z_{new}, \calL,  \pi^*(l')}$ which proves the statement in this case.
It means, that we can assume, that for every $1 \leq j \leq m$ and $i \in B_j$ we have $l_j \geq E_{w_{i, k_i - 1}}$.

Notice, that $( \pi^*(l'), E_u) = 0$ for all vertices $u \in (\calv_{new} \setminus \calv_{1})$ and $|l_j|$ contains a vertex in every component of $| Z''_j| \cap (\calv_{new} \setminus \calv_{1})$, which means that
 $|l_j|$ contains evey component of $| Z''_j| \cap (\calv_{new} \setminus \calv_{1})$.

This means, that $(\calv_{new} \setminus \calv_{1})_i \in |l_j|$ for every index $i \in B_j$, let's denote $\sum_{u \in (\calv_{new} \setminus \calv_{1})_i , i \in B_j} E_u = E_j$, then we get, that $h^1(Z''_j,  \mfl_j) =  h^1(Z''_j - E_j,  \mfl_j- E_j) + \chi(\pi^*(l')) - \chi(\pi^*(l') + E_j)$.

The same proof as in the main theorem about cohomology of natural line bundles on relatively generic singularities from \cite{R} gives, that $h^1(Z''_j - E_j, D- E_j) = h^1(Z''_j - E_j, \calL_{gen})$, where $\calL_{gen} \in r^{-1}(\calO_{(Z''_j)_1}(D - E_j)) \subset \pic^{\pi^*(l')- E_j}(Z''_j - E_j)$ is a generic line bundle.

Now if $\calL'_{gen} \in  r^{-1}(\calO_{(Z''_j)_1}(D)) \subset \pic^{\pi^*(l')}(Z''_j)$ is a generic line bundle then we have $h^1(Z''_j, \calL'_{gen}) \geq h^1(Z''_j - E_j, \calL'_{gen}- E_j) + \chi(\pi^*(l')) - \chi(\pi^*(l') + E_j)$.

This means, that $h^1(Z''_j, \calL'_{gen})  \geq h^1(Z''_j,  \mfl_j)$, however we have obviously $h^1(Z''_j,  \mfl_j) \geq h^1(Z''_j, \calL'_{gen})$ by semicontinuity, so we get
$h^1(Z''_j,  \mfl_j)  = h^1(Z''_j, \calL'_{gen})$ and this means, that $ T(Z''_j, \mfl_{u}, \pi^*(l')) = h^1(Z''_j, \calL'_{gen}) + 1$.

On the other hand notice, that $h^1(Z''_j, \calL'_{gen}) = h^1(Z''_j - E_j, \calL'_{gen}- E_j) + \chi(\pi^*(l')) - \chi(\pi^*(l') + E_j)$, so we have $H^0(Z''_j, \calL'_{gen})_0 = \emptyset$,
which means, that $D(Z''_j, \calL, \pi^*(l')) = 1$.
We finally get, that $T(Z''_j, \mfl_{u}, \pi^*(l')) = T(Z''_j, \calL, \pi^*(l')) $.

This means, that we have $h^1(Z_{new}, \mfl_{new}) = \sum_{1 \leq j \leq m} T(Z''_j, \calL, \pi^*(l')) \leq b_{Z_{new}, \calL,  \pi^*(l')}$, which proves part 3)
in the special case when $(l', E_u) = 0$ for every $u \in \calv_2$.

In the following we prove part 3) in full generality.

So let's denote the $*$-restriction of the Chern class $l'$ to the vertex set $\calv_2$ by $l'_2$ and to the vertex set $\calv_1$ by $l'_1$ and let's have a generic divisor $D \in (c^{l'_1})^{-1}(\calL) \in \eca^{l'_1}(Z_1)$ and a generic divisor $D' \in \eca^{l'_2}(Z)$, then we know, that $D' + D$ gives a generic line bundle $\mfl \in c^{l'}(\eca^{ l', \calL}(Z))$, this means, that 
we have to compute $h^1(Z, \mfl) = h^1(Z, D + D')$.

Let's recall the main theorem about cohomology of twisted line bundles from \cite{NNAD}:

\begin{corollary}\label{cor:formL}
Let's have an arbitrary singularity with resolution $\tX$, a Chern class $l'\in -\calS'$, an effective cycle $Z\geq 0$ and a line bundle $\calL_0$ with $H^0(Z,\calL_0)_{reg}\not=\emptyset$.
Furthermore let's have a generic line bundle $\calL^{im}_{gen}\in \im (c^{l'}(Z))$, then we have:
$$h^1(Z,\calL_0\otimes\calL^{im}_{gen})=
\max_{0\leq Z_1\leq Z} \{\, h^1(Z_1,\calL_0)- (l', Z_1)\}.
$$
\end{corollary}

It means, that we have the following equation in our case:

\begin{equation}
h^1(Z, D + D') = \max_{0 \leq Z' \leq Z}( h^1(Z', D) - (Z', l'_2)) =  \max_{0 \leq Z' \leq Z} \left( \sum_{1 \leq i \leq n}h^1(Z'_i, D) - (Z'_i, l'_2) \right),
\end{equation}
where the connected components of the cycle $Z'$ are $Z'_1, \cdots, Z'_n$.

So assume in the following, that $h^1(Z, D + D') = \sum_{1 \leq i \leq n}h^1(Z'_i, D) - (Z'_i, l'_2)$ for some cycle $Z'$ and assume first that $n$ is minimal with that property and second that
if the number of components is $n$, then $Z'$ is maximal with that property, in the sense, there isn't a cycle $Z \geq Z^* > Z'$ with the same property.

Now let the components of $Z_1$ be $Z_{1, 1}, \cdots, Z_{1, k}$, we claim, that if $1 \leq i \leq n$ and $1 \leq j \leq k$ and $|Z'_i| \cap |Z_{1, j}| \neq \emptyset$
then $Z'_i \geq Z_{1, j}$.

Assume to the contrary that there exists $i, j$ such that $|Z'_i| \cap |Z_{1, j}| \neq \emptyset$ and $Z'_i \ngeq Z_{1, j}$.

Now let's have $Z'' = \max(Z', Z_{1, j})$ and notice that $ h^1(Z'', D) - (Z'', l'_2) \geq  h^1(Z', D) - (Z', l'_2)$, which means, that $h^1(Z, D + D') = h^1(Z'', D) - (Z'', l'_2)$
however this contradicts to the minimality of $n$ or the maximality of $Z'$ so we got a contradiction.

Now for every index $1 \leq i \leq n$, $(Z'_i)_1$ is a sum of a few components of $Z_1$ which means, that the divisor $D$ is a generic divisor of the line bundle $\calO_{(Z'_i)_1}(D)$.

We use the part 3) in the case already proved to get for each $1 \leq i \leq n$ cycles $Z''_{i} \leq Z'_i$ with connected components $Z''_{i, 1}, \cdots Z''_{i, c_i}$, such that for every $1 \leq i \leq n$ we have:

\begin{equation}
h^1(Z'_i, D) = \sum_{1 \leq j \leq c_i} T(Z''_{i, j}, \calL, l'_1).
\end{equation}

Notice that we can assume, that for all $1 \leq i \leq n, 1 \leq j \leq c_i$ we have $ D(Z''_{i, j}, \calL, l'_1) = 1$ or $Z''_{i, j} = (Z''_{i, j})_1$.

Indeed if $D(Z''_{i, j}, \calL, l'_1) = 0$, then we have $T(Z''_{i, j}, \calL, l'_1) = h^1((Z''_{i, j})_1, \calL)$ so we can replace $Z''_{i, j}$ by $(Z''_{i, j})_1$.

Notice that we have $h^1(Z, D + D') = h^1(Z', D + D') = \sum_{1 \leq i \leq n}h^1(Z'_i, D) - (Z'_i, l'_2)$, which means, that $h^1(Z, D + D') = \sum_{1 \leq i \leq n, 1 \leq j \leq c_i} T(Z''_{i, j}, \calL, l'_1) - (Z'_{i, j}, l'_2) \leq \sum_{1 \leq i \leq n, 1 \leq j \leq c_i} h^1(Z''_{i, j}, D + D')$.

It means, that we have $h^1(Z, D + D') = \sum_{1 \leq i \leq n, 1 \leq j \leq c_i} h^1(Z''_{i, j}, D + D')$ and $h^1(Z''_{i, j}, D + D') = T(Z''_{i, j}, \calL, l'_1) - (Z'_{i, j}, l'_2)$ for every $ 1 \leq i \leq n, 1 \leq j \leq c_i$.

Now if $D(Z''_{i, j}, \calL, l') = 0$, then let's denote $Z'_{i, j} = (Z''_{i, j})_1$ and if $D(Z''_{i, j}, \calL, l') = 1$, then let's have $Z'_{i, j} = Z''_{i, j}$, we claim, that:

\begin{equation}
h^1(Z, D + D') \leq \sum_{1 \leq i \leq n, 1 \leq j \leq c_i}T(Z'_{i, j}, \calL, l') \leq b_{Z, \calL, l'}.
\end{equation}

For this we are enough to prove, that $T(Z'_{i, j}, \calL, l') \geq h^1(Z''_{i, j}, D + D')$.

If $D(Z''_{i, j}, \calL, l') = 0$, then $T(Z'_{i, j}, \calL, l') = h^1((Z''_{i, j})_1, D + D')$ and $h^1(Z''_{i, j}, D + D') = h^1((Z''_{i, j})_1, D + D')$ by the relative dominancy,
so we indeed get $T(Z'_{i, j}, \calL, l') = h^1(Z''_{i, j}, D + D') $.

So assume, that $D(Z''_{i, j}, \calL, l') = 1$, in this case we have $Z'_{i, j} = Z''_{i, j}$
and $T(Z'_{i, j}, \calL, l') = g(Z''_{i, j}, \calL, l') + 1$ and $h^1(Z''_{i, j}, D + D') = g(Z''_{i, j}, \calL, l'_1) - (Z''_{i, j}, l'_2) + 1$.

It means, that we have to prove $g(Z''_{i, j}, \calL, l') \geq g(Z''_{i, j}, \calL, l'_1) - (Z''_{i, j}, l'_2)$.

Indeed let's have a generic line bundle $\calL_{gen} \in r^{-1}(\calL) \subset \pic^{l'}(Z''_{i, j})$, then $\calL_{gen} - D'$ is a generic line bundle in $ r^{-1}(\calL) \subset \pic^{l'}(Z''_{i, j})$, 
so we have $g(Z''_{i, j}, \calL, l') = h^1(Z''_{i, j}, \calL_{gen})$ and $g(Z''_{i, j}, \calL, l'_1)  =  h^1(Z''_{i, j}, \calL_{gen}- D')$ and we have the exact sequence:

\begin{equation}
0 \to H^0(Z''_{i, j}, \calL_{gen}- D') \to H^0(Z''_{i, j}, \calL_{gen}) \to H^0(D', \calL_{gen})  \to H^1(Z''_{i, j}, \calL_{gen}- D') \to H^1(Z''_{i, j}, \calL_{gen}) \to 0.
\end{equation}

This yields the inequailtiy $h^1(Z''_{i, j}, \calL_{gen}) \geq h^1(Z''_{i, j}, \calL_{gen}- D') -  (Z'_{i, j}, l'_2)$ which proves our theorem completely.
\end{proof}

\section{Dimensions of images of Abel maps for relatively generic singularities}

In \cite{NNAD} the authors investigated images of Abel maps for surface singularities, and gave two different algorithms, which can compute these invariants using cohomology of several 
cycles or the resolution of periodic constants of singularities given by blowing up the original singularity in generic points.

One of the main theorems which gives the dimension of images of Abel maps was Theorem \ref{dim}.

Now assume, that we have $2$ resolution graphs $\mathcal{T}_1 \subset \mathcal{T}$ and let $\tX_1$ be a fixed singularity with resolution graph $\mathcal{T}_1$.

Furthermore for every vertex $v \in \calv \setminus \calv_1$ we fix cuts $D_v$ on $\tX_1$ and  let $\tX_1 \subset \tX$ be a singularity with resolution graph $\mathcal{T}$, 
such that we glue the exceptional divisors $E_v$ along $D_v$.

 If $\tX$ is a relatively generic singularity corresponding to $\tX_1$ and the cuts $D_v$, then we can compute the cohomology of the cycles $h^1(Z_1)$, where $0\leq Z_1 \leq Z$.

Indeed, let the connected components of the cycle be $Z_{1, 1}, \cdots, Z_{1, j}$, then we have $h^1(Z_1) = \sum_{1 \leq i \leq j} h^1(Z_{1, i})$ and
we have $ h^1(Z_{1, i}) =  h^1(Z_{1, i} - E_{|Z_{1, i}|}, - E_{|Z_{1, i}|})$.

On the other hand we can use  Theorem \ref{rell} above about cohomology of natural line bundles on relatively generic surface singularties from \cite{R} to conclude, that:

\begin{equation}
h^1(Z_{1, i}) = \chi(-l'_m) - \min_{0 \leq l \leq Z_{1, i} - E_{|Z_{1, i}|} }(\chi(-l'_m+l)-  h^1((Z_{1, i} - E_{|Z_{1, i}|}  -l)_1,-R(E_{|Z_{1, i}|} + l)) ),                                   
\end{equation}
where $l'_m = c^1( R(- E_{|Z_{1, i}|} ))$.

It means, that in the case of relatively generic singularities we get combinatorial formulas for the dimensions of images of Abel maps using only analytic invariants of the base singularity $\tX_1$.

However in the following we would like to get another formula following more the philosophy of the case of generic analytic structures from \cite{NNAD}.

The coincidence of the two formulae is mysterious and cannot be seen directly by the author at this point, and it may somehow be coded in the combinatorics of the possible cohomology numbers of
the base singularity $\tX_1$ for its different analytic types.

\begin{theorem}
Let's have the setup as above, so $\mathcal{T}_1 \subset \mathcal{T}$ be $2$ resolution graphs and let $\tX_1$ be a fixed singularity with resolution graph $\mathcal{T}_1$.
Furthermore for every vertex $v \in \calv \setminus \calv_1$ we fix cuts $D_v$ on $\tX_1$ and  let $\tX_1 \subset \tX$ be a singularity with resolution graph $\mathcal{T}$, 
such that we glue the exceptional divisors $E_v$ along $D_v$.

For an arbitrary element $l' \in - S'(\mathcal{T})$  and cycle $B$ on $\tX$ let's denote $g(B, l')= h^1(B, \calL)$, where $\calL$ is a generic line bundle in $r^{-1}(\calL_1)$ and
$\calL_1$ is a generic line bundle in $\im(c^{R(l')}(B_1))$.

Notice, that by the Theroem \ref{rel} on relatively generic line bundles form \cite{R}, this is an analytic invariant depending just on $\tX_1$ and the cuts $D_v$.

Now, let $l' \in -S'$ be arbitrary element, $Z$ an arbitrary cycle on $\tX$ and let's denote:
\begin{equation}
b_{Z, l'} = \max_{Z' \leq Z} \left( \sum_{1 \leq i \leq n} T(Z'_i, l')   \right).
\end{equation}
Here $|Z'_1|, \cdots, |Z'_n|$ are the connected components of $|Z'|$ with $\sum_{1 \leq i \leq n}Z'_i = Z'$ and $T(Z'_i, l') = g(Z'_i, l') + D(Z'_i, l')$, where 
$D(Z'_i, l') = 1$ if $(l', \calL_1)$ is not relative dominant on $Z'_i$, where $\calL_1$ is a generic line bundle in $\im(c^{R(l')}((Z'_i)_1))$, and $D(Z'_i, l') = 0$ otherwise.

1)  We have the inequality $\dim(\im(c^{l'}(Z)) \leq h^1(Z) - b_{Z, l'}$.

2) If $\tX$ is a relatively generic singularity corresponding to $\tX_1$ and the cuts $D_v$, then we have the equality $\dim(\im(c^{l'}(Z)) = h^1(Z) - b_{Z, l'}$
\end{theorem}
\begin{remark}
One could prove this theroem using our main theorem about dimensions of relative Abel maps for the case when the fixed line bundle on $\tX_1$ is a generic line bundle 
$\calL_1 \subset \im(c^{R(l')}(B_1))$. The only difficulty would be to deal with the cuts $D_v, v \in \calv \setminus \calv_1$, but this could be overcome by blowing up the singularity
sequentially at the intersection points of $D_v$ and the exceptional divisors $E_v$.
However we choose to prove the theorem in a different way which is more similar to the original treatment of dimensions of Abel images in \cite{NNAD}.
\end{remark}

\begin{proof}

For part 1) notice, that for a generic line bundle $\calL \in \im(c^{l'}(Z))$ one has $h^1(Z, \calL) = h^1(Z) - \dim(\im(c^{l'}(Z))$.

Now for all $Z' \leq Z$ one has $h^1(Z, \calL) \geq h^1(Z', \calL) = \sum_{1 \leq i \leq n} h^1(Z'_i, \calL)$, so we only need to prove $h^1(Z'_i, \calL) \geq g(Z'_i, l') + D(Z'_i, l')$.

We know that $\calL | (Z'_i)_1$ is a generic line bundle in $ \im(c^{R(l')}((Z'_i)_1))$, so we immediately get, that $h^1(Z'_i, \calL) \geq g(Z'_i, l')$.

Furthermore if $(l', \calL| (Z'_i)_1)$ is not relatively dominant on $Z'_i$, then the generic line bundle in $r^{-1}(\calL | (Z'_i)_1)$ is not in $\im(c^{l'}(Z'_i))$, while $\calL| Z'_i \in \im(c^{l'}(Z_i))$, which means, that 
$h^1(Z'_i, \calL) \geq g(Z'_i, l') + 1$, so we are done with part 1).

Now for part 2) we want to use the algorithm from  \cite{NNAD} computing the dimensions of images of Abel maps, so let's recall the notations, the setup and the algorithm:

Fix some resolution $\tX$ of $(X,o)$ and
 $-l' = \sum_{v \in \calv} a_v E_v^*\in\calS'\setminus \{0\}$ (hence each $a_v\in\Z_{\geq 0}$) and an effective cycle $Z$ on the resolution.
We will consider a finite sequence  of blowing ups starting from $\tX$.
In order to find a bound for the number of blowing ups recall that for any representative $\omega$ in
$H^0(\tX\setminus E,\Omega^2_{\tX})/H^0(\tX,\Omega^2_{\tX})$ the order of pole of $\omega$
along some $E_v$ is less than or equal to
 the $E_v$--multiplicity $m_v$ of $\max\{0,\lfloor Z_K\rfloor\} $ (see e.g.
\cite[7.1.3]{NNA1}).

Then, for every $v\in \calv$ with $a_v>0$ we fix $a_v$ generic points on $E_v$, say
$p_{v,k_v}$, $1\leq k_v\leq a_v$. Starting from each $p_{v,k_v}$ we consider a sequence of blowing ups  of length  $m_v$:
first we blow up $p_{v,k_v}$ and we create the exceptional curve $F_{v,k_v,1}$,
then we blow up a generic point of
$F_{v,k_v,1}$ and we create $F_{v,k_v,2}$, and we do this all together $m_v$ times.

We proceed in this way with all points $p_{v,k_v}$, hence we get $\sum_va_v$ chains of modifications.
If $a_vm_v=0$ we do no modification along $E_v$.
A set of integers
${\bf s}=\{{\bf s}_{v,k_v}\}_{v\in\calv,\ 1\leq k_v\leq a_v}$ with $0\leq {\bf s}_{v,k}\leq m_v$ provides an
intermediate step of the tower: in the $(v,k_v)$ tower we do exactly ${\bf s}_{v,k_v}$ blowing ups;
${\bf s}_{v,k_v}=0$ means that we do not blow up $p_{v,k_v}$ at all.
(In the sequel, in order to avoid aggregation of indices, we simplify $k_v$ into $k$.)
Let us denote this modification by $\pi_{{\bf s}}:\tX_{{\bf s}}\to \tX$.
In $\tX_{{\bf s}}$ we find the exceptional curves $\cup_{v\in\calv}E_v\cup \cup_{v,k}
\cup_{1\leq t \leq {\bf s}_{v,k}}F_{v,k,t}$; we index the set of vertices as
$\calv_{{\bf s}}:=\calv \cup \cup_{v,k}
\cup_{1\leq t\leq {\bf s}_{v,k}}\{w_{v,k,t}\}$.
At each level ${\bf s}$ we set the next objects: $Z_{{\bf s}}:=\pi_{{\bf s}}^*(Z)$,
$I_{{\bf s}}:=\cup_{v,k}\{w_{v,k,{\bf s}_{v,k}}\}$, $-l'_{{\bf s}}:= \sum_{v,k}F^*_{v,k,{\bf s}_{v,k}}$
(in $L'_{{\bf s}}$, where
 $F_{v,k,0}=E_v$), $d_{{\bf s}}:=\dim \im c^{l'_{{\bf s}}}(Z_{{\bf s}})=d_{Z_{{\bf z}}}(I_{{\bf s}})$
and  $e_{{\bf s}}:=e_{Z_{{\bf z}}}(I_{{\bf s}})$
(both considered in  $\tX_{{\bf s}}$), one has
 again $d_{{\bf s}}\leq e_{{\bf s}}$ for any ${\bf s}$.

From definitions, for ${\bf s}={\bf 0}$ one has
$I_{{\bf 0}}= |l'|$, $e_{{\bf 0}}=e_Z(l')$ and $d_{{\bf 0}}=d_Z(l') := \dim(\im(c^{l'}(Z))$.

There is a natural partial ordering on the set of ${\bf s}$--tuples. Some of the above invariants are
constant with respect to ${\bf s}$, some of them are only monotonous. E.g., by Leray spectral sequence
one has $h^1(\calO_{Z_{{\bf s}}})=h^1(\calO_Z)$ for all ${\bf s}$. One the other hand,
\begin{equation}\label{eq:ineqes}\mbox{
if ${\bf s}_1\leq {\bf s}_2$ then $e_{{\bf s}_1}=h^1(\calO_{Z_{{\bf s}_1}})-\dim
\Omega_{Z_{{\bf s}_1}}(I_{{\bf s}_1})\geq  h^1(\calO_{Z_{{\bf s}_2}})-\dim
\Omega_{Z_{{\bf s}_2}}(I_{{\bf s}_2})=  e_{{\bf s}_2}$}\end{equation}
because
$\Omega_{Z_{{\bf s}_1}}(I_{{\bf s}_1})\subset \Omega_{Z_{{\bf s}_2}}(I_{{\bf s}_2})$.
In fact, for any
$\omega$, the pole--order along $F_{v,k,{\bf s}_{v, k}+1}$ of its pullback is one less than the pole--order of
$\omega$ along  $F_{v,k,{\bf s}_{v, k}}$. Hence, for ${\bf s}={\bf m}$
(that is, when ${\bf s}_{v,k}=m_v$ for all $v$ and $k$, hence all the possible pole--orders along
$I_{{\bf m}}$ automatically vanish) one has
$\Omega_{Z_{{\bf m}}}(I_{{\bf m}})=H^0(\tX_{{\bf m}}, \Omega^2_{{\tX_{{\bf m}}}}(Z_{{\bf m}}))/
H^0(\Omega^2_{{\tX_{{\bf m}}}})$. Hence $e_{{\bf m}}=0$. In particular,  necessarily
$d_{{\bf m}}=0$ too.

More generally, for any ${\bf s}$ and $(v,k)$ let ${\bf s}^{v,k}$ denote that  tuple which
is obtained from
${\bf s}$ by increasing ${\bf s}_{v,k}$ by one. By the above discussion
if   no form has pole along $F_{v,k,{\bf s}}$ then
$\Omega_{Z_{{\bf s}}}(I_{{\bf s}})=\Omega_{Z_{{\bf s}^{v,k}}}(I_{{\bf s}^{v,k}})$, hence
$e_{{\bf s}}=  e_{{\bf s}^{v,k}}$. Furthermore, by Laufer duality (or,
integral presentation of the Abel map as in \cite[\S 7]{NNA1}), under such condition
$d_{{\bf s}}=  d_{{\bf s}^{v,k}}$ as well.

Therefore, we can redefine $e_{{\bf s}}$ and  $d_{{\bf s}}$  for tuples
${\bf s}=\{{\bf s}_{v,k}\}_{v,k}$ even for
arbitrary ${\bf s}_{v,k}\geq 0$: $e_{{\bf s}} = e_{\min\{{\bf s},{\bf m}\}}$ and
$d_{{\bf s}} = d_{\min\{{\bf s},{\bf m}\}}$ (and these values agree with the ones which might be
obtained by the
 first original
construction  applied for larger chains of blow ups).

The next theorem from \cite{NNAD} relates the invariants  $\{d_{{\bf s}}\}_{{\bf s}}$ and
$\{e_{{\bf s}}\}_{{\bf s}}$.

\begin{theorem}
 With the above notations  the following facts hold.

(1) $ d_{{\bf s}} -  d_{{\bf s}^{v, k}} \in \{0, 1\}$.

(2) If for some fixed ${\bf s}$ the numbers $\{d_{{\bf s}^{v, k}}\}_{v,k}$ are not the same,
then $d_{{\bf s}} = \max_{v, k}\{\,d_{{\bf s}^{v, k}}\}$.
 In the case when all the numbers $\{d_{{\bf s}^{v, k}}\}_{v,k}$ are the same,
 then if this common value $d_{{\bf s}^{v, k}}$ equals $e_{{\bf s}}$, then $d_{{\bf s}} =
e_{{\bf s}}  =d_{{\bf s}^{v, k}}$;  otherwise $d_{{\bf s}} = d_{{\bf s}^{v, k}}+1$.

(3) There is a a sequence ${\bf s}_1, \cdots {\bf s}_N$, such, that ${\bf s}_1 = 0$ and ${\bf s}_{i+1} = {\bf s}_i^{v_i, k_i}$ for some vertex $v_i$ and $1 \leq k_i \leq a_{v_i}$, such that $ d_{{\bf s}_i} = d_{{\bf s}_{i+1}} + 1$ for
all $1 \leq i \leq N-1$, and $d_{{\bf s}_N} = d_{{\bf s}_N^{v, k}}$ for all $v, k$, which means $d_{{\bf s}_N} = e_{{\bf s}_N}$.
\end{theorem}

Now let's return to the proof of the theorem, we will prove the statement by a downgoing  induction on $i$, where $0 \leq i \leq N$.
If $i = N$, then we have $d_{{\bf s}_N} = e_{{\bf s}_N}$, which means, that $h^1(Z_{{\bf s}_N}) - d_{{\bf s}_N} = h^1(Z_{{\bf s}_N}) -  e_{Z, {\bf s}_N} = h^1( (Z_{{\bf s}_N})_{\calv \setminus I})$, where 
$I = |l'_{{\bf s}_N}|$.

Now let's have $(Z_{{\bf s}_N})_{\calv \setminus I} = \sum_{1 \leq i \leq n} Z'_i$, where $|Z'_i|$ are the connected components of $|(Z_{{\bf s}_N})_{\calv \setminus I} |$, now we have 
$h^1(Z_{{\bf s}_N}) - d_{{\bf s}_N} = \sum_{1 \leq i \leq n} h^1(Z'_i)$.

Notice, that by Theorem \ref{rell} from \cite{R} about cohomology of natural line bundles on relatively generic singularities, we easily get:
\begin{equation}
h^1(Z'_i) = T(Z'_i, l'_{{\bf s}_N}) = g(Z'_i, l'_{{\bf s}_N}) + D(Z'_i, l'_{{\bf s}_N}).
\end{equation}
So it means, that:

\begin{equation}
h^1(Z_{{\bf s}_N}) - d_{{\bf s}_N} =  \sum_{1 \leq i \leq n} T(Z'_i, l'_{{\bf s}_N}) \leq b_{Z_{{\bf s}_N}, l'_{{\bf s}_N}}.
\end{equation}
On the other hand we always have $h^1(Z_{{\bf s}_N}) - d_{{\bf s}_N} \geq b_{Z_{{\bf s}_N}, l'_{{\bf s}_N}}$, so we have $h^1(Z_{{\bf s}_N}) - d_{{\bf s}_N} = b_{Z_{{\bf s}_N}, l'_{{\bf s}_N}}$, which means, that the case $i = N$ is proved.

In the following we will prove, that $ b_{Z_{{\bf s}_{i+1}}, l'_{{\bf s}_{i+1}}} \leq b_{Z_{{\bf s}_i}, l'_{{\bf s}_i}} + 1$:

Let the new vertex of $\mathcal{T}_{{\bf s}_{i+1}}$ be $u$ and its neighbour $w$, and assume, that $b_{Z_{{\bf s}_{i+1}}, l'_{{\bf s}_{i+1}}} = \sum_{1 \leq j \leq n} T(Z'_j, l'_{{\bf s}_{i+1}})$, and if there is a component $Z'_j$, such that $w \in |Z'_j|$, then $j = 1$.

Now if there is no component $Z'_j$, such that $w \in |Z'_j|$, then the statement is trivial, since in this case $T(Z'_j, l'_{{\bf s}_{i+1}}) = T(Z'_j, l'_{{\bf s}_{i}})$ for all components which are not the one element $u$, and there can be a component
$Z'_j$, such that $| Z'_j| = u$, but then simply we have $T(Z'_j, l'_{{\bf s}_{i+1}}) = 0$, so in this case we get $ b_{Z_{{\bf s}_{i+1}}, l'_{{\bf s}_{i+1}}} = b_{Z_{{\bf s}_i}, l'_{{\bf s}_i}}$.

So we can assume, that $ w \in |Z'_1|$, now similarly we get $T(Z'_j, l'_{{\bf s}_{i+1}}) = T(Z'_j, l'_{{\bf s}_{i}})$ for all $ j > 1$.

Let's have $Z'_1 = a E_u + b E_w + A$, where $|A|$ doesn't contain $u, w$.

We claim, that $T(Z'_1, l'_{{\bf s}_{i+1}}) \leq T(A + b E_w, l'_{{\bf s}_{i}}) + 1$, where on the right hand side we take the cycle $A + b E_w$ on $\tX_{{\bf s}_{i}}$.

First notice, that we can assume, that the coefficient $a$ equals $(Z_{{\bf s}_{i+1}})_u$, because if we increase $a$, the value $T(Z'_1, l'_{s_{i+1}})$ can only increase.

Now, let's have a generic line bundle $\calL_1 \in \im( c^{ R(l'_{{\bf s}_{i}})} ((A + b E_w)_1))$ and a generic line bundle $\calL \in r^{-1}(\calL_1)$, now we have $h^1(A + b E_w, \calL) = g(A + b E_w, l'_{{\bf s}_{i}})$.

In the case $w$ is a vertex of the base singularity, then we can also assume, that $\calL_1$ has got a section whose divisor is a generic divisor $D$ in $\eca^{R(l'_{{\bf s}_{i}})}((A + b E_w)_1)$ and we blow up $E_w$ in $1$
intersection point of $E_w$ and $D$ to get the new exceptional divisor $E_v$.

We have the pullback maps $\pi^{*}_1 : \pic^{R(l'_{{\bf s}_{i}})}((A + b E_w)_1)  \to \pic^{ \pi^*(R(l'_{{\bf s}_{i}}) )}( (Z'_1)_1)$
and $\pi^{*}_2 : \pic^{R(l'_{{\bf s}_{i}})}(A + b E_w)  \to \pic^{ \pi^*(R(l'_{{\bf s}_{i}}) )}(Z'_1)$ which are both isomorphisms, since by the Leray spectral sequences
we have $h^1(A + b E_w) = h^1(Z'_1)$ and $h^1((A + b E_w)_1) = h^1( (Z'_1)_1)$.

Now notice, that $\pi_1^*(\calL) - E_u |  (Z'_1)_1$ is a generic line bundle in $\im \left(c^{ R(l'_{{\bf s}_{i+1}})} ((Z'_1)_1)  \right)$, and $\pi_2^*(\calL) - E_u$ is a generic line bundle
in $r^{-1}( \pi_1^*(\calL) - E_u |  (Z'_1)_1)$, so it means, that $g(Z'_1, l'_{{\bf s}_{i+1}}) = h^1(Z'_1 , \pi_2^*(\calL) - E_u)$.

We can easily see, that $  h^1(Z'_1 , \pi_2^*(\calL) - E_u) \leq h^1(A + b E_w, \calL) + 1$, which means at least, that $g(A + b E_w, l'_{{\bf s}_{i}}) + 1 \geq g(Z'_1, l'_{{\bf s}_{i+1}})$.

Indeed we have the exact sequence:

\begin{equation*}
H^0(E_u, \pi_2^*(\calL) ) \to H^1(Z'_1 , \pi_2^*(\calL) - E_u) \to H^1(Z'_1  + E_u, \pi_2^*(\calL)) \to 0,
\end{equation*}
which using the facts, that $h^0(E_u, \pi_2^*(\calL) ) = 1$ and $h^1(Z'_1  + E_u, \pi_2^*(\calL)) =  h^1(A + b E_w, \calL)$ yields the statement.

Now, if $D(A + b E_w, l'_{{\bf s}_{i}}) = 1$, then we are done, so we can assume, that $D(A + b E_w, l'_{{\bf s}_{i}}) = 0$.
This means, that $(l'_{{\bf s}_i}, \calL_{1})$ is relative dominant on $A + b E_w$, where $\calL_{1}$ is a generic line bundle in $\im(c^{R(l'_{{\bf s}_i})}((A + b E_w)_1))$, which means, that:
\begin{equation*}
T(A + b E_w, l'_{{\bf s}_{i}})= g(A + b E_w, l'_{s_{i}}) = h^1(A + b E_w) - d_{A + b E_w, l'_{{\bf s}_{i}}}.
\end{equation*}

On the other hand we have:

\begin{equation*}
h^1(Z_{{\bf s}_{i+1}}) - d_{{\bf s}_{i+1}} =  b_{Z_{{\bf s}_{i+1}}, l'_{{\bf s}_{i+1}}}  =  \sum_{1 \leq j \leq n} T(Z'_j, l'_{{\bf s}_{i+1}}) \leq \sum_{1 \leq j \leq n} \left(  h^1(Z'_j) - d_{Z'_j, l'_{{\bf s}_{i+1}}} \right)
\end{equation*}

However we know, that $h^1(Z_{{\bf s}_{i+1}}) - d_{{\bf s}_{i+1}} \geq \sum_{1 \leq j \leq n} \left(  h^1(Z'_j) - d_{Z'_j, l'_{{\bf s}_{i+1}}} \right)$, so we get, that:

\begin{equation*}
 h^1(Z'_1) - d_{Z'_1, l'_{{\bf s}_{i+1}}} = T(Z'_1, l'_{{\bf s}_{i+1}}).
\end{equation*}

So now we have $T(Z'_1, l'_{{\bf s}_{i+1}}) - T(A + b E_w, l'_{{\bf s}_{i}}) =  d_{A + b E_w, l'_{{\bf s}_{i}}} - d_{Z'_1, l'_{{\bf s}_{i+1}}} \leq 1$, so we have proved, that $ b_{Z_{{\bf s}_{i+1}}, l'_{{\bf s}_{i+1}}} \leq b_{Z_{{\bf s}_i}, l'_{{\bf s}_i}} + 1$.

Now, by induction we know, that $b_{Z_{{\bf s}_{i+1}}, l'_{{\bf s}_{i+1}}} = h^1 (Z_{{\bf s}_{i+1}}) - d_{{\bf s}_{i+1}}$ and $ h^1 (Z_{{\bf s}_{i}}) - d_{{\bf s}_{i}} + 1 = h^1 (Z_{{\bf s}_{i+1}}) - d_{{\bf s}_{i+1}}$, which means, that:

\begin{equation}
 b_{Z_{{\bf s}_i}, l'_{s_i}} \geq h^1 (Z_{{\bf s}_{i}}) - d_{{\bf s}_{i}}.
\end{equation}

However by part 1) we always have $b_{Z_{{\bf s}_i}, l'_{{\bf s}_i}} \leq h^1 (Z_{{\bf s}_{i}}) - d_{{\bf s}_{i}}$, which means $b_{Z_{{\bf s}_i}, l'_{{\bf s}_i}} = h^1 (Z_{{\bf s}_{i}}) - d_{{\bf s}_{i}}$, and we are done.
\end{proof}


\begin{thebibliography}{30}



\bibitem[A62]{Artin62} Artin, M.:
Some numerical criteria for contractibility of curves on algebraic surfaces.
{\em  Amer. J. of Math.}, {\bf 84}, 485-496, 1962.

\bibitem[A66]{Artin66} Artin, M.:
On isolated rational singularities of surfaces.
{\em Amer. J. of Math.}, {\bf 88}, 129-136, 1966.



\bibitem[BN10]{BN} Braun, G. and N\'emethi, A.:
Surgery formula for Seiberg--Witten invariants of negative definite plumbed 3--manifolds,
{ \em J. f\"ur die reine und ang. Math.} {\bf 638} (2010), 189--208.

\bibitem[BN07]{BNnewt} Braun, G. and N\'emethi, A.:
Invariants of Newton non-degenerate surface singularities, {\em Compositio Math.} {\bf 143} (2007), 1003--1036.

\bibitem[BV99]{BV} Brion, M. and Vergne, M.: Arrangement of hyperplanes. I: Rational functions and the Jeffrey--Kirwan residue,
{\em Ann. Sci. l’École Norm. Sup.} {\bf 32} (1999), no. 5, 715--741.

\bibitem[CDGZ04]{CDGPs} Campillo, A.,  Delgado, F. and Gusein-Zade, S. M.:
Poincar\'e series of a rational surface singularity, {\em Invent. Math.} {\bf 155} (2004),
no. 1, 41--53.

\bibitem[CDGZ08]{CDGEq}  Campillo, A.,  Delgado, F. and Gusein-Zade, S. M.:
Universal abelian covers of rational
surface singularities and multi-index filtrations,
{\em Funk. Anal. i Prilozhen.} {\bf 42} (2008), no. 2, 3--10.

\bibitem[CHR03]{CHR} Cutkosky, S. D.,  Herzog, J. and Reguera, A.:
Poincar\'e series of resolutions of surface singularities,
{\em Trans. of the AMS} {\bf 356} (2003), no. 5, 1833--1874.

\bibitem[EN85]{EN} Eisenbud, D. and Neumann, W.: Three--dimensional link theory and invariants of plane curve singularities,
{\em Princeton Univ. Press} (1985).

\bibitem[GS99]{GS} Gompf, R.E. and Stipsicz, A.:  An introduction to
$4$--manifolds and Kirby calculus, {\em Graduate Studies in
Mathematics} {\bf 20} (1999), Amer. Math. Soc.

\bibitem[GR70]{GR} Grauert, H., Remmert, R.:
Coherent analytic sheaves, Grundlehren der mathematischen Wissenschaften, 1984 Springer ??????

\bibitem[GrRie70]{GrRie} Grauert, H. and Riemenschneider, O.: Verschwindungss\"atze f\"ur analytische
kohomologiegruppen auf komplexen R\"aumen, {\it Inventiones math.} {\bf 11} (1970), 263-292.



\bibitem[Ha77]{hartshorne}  Hartshorne, R.:  Algebraic Geometry, Graduate Texts
in Mathematics {\bf 52}, Springer-Verlag, 1977.




\bibitem[L13]{LPhd} L\'aszl\'o, T.: Lattice cohomology and Seiberg--Witten invariants of normal surface singularities, PhD. thesis,
Central European University, Budapest, 2013.

\bibitem[LN14]{LN} L\'aszl\'o, T. and N\'emethi, A.: Ehrhart theory of polytopes and Seiberg-Witten invariants of plumbed 3--manifolds,
{\em Geometry and Topology} {\bf 18} (2014), no. 2, 717--778.

\bibitem[LN15]{LNRed} L\'aszl\'o, T. and N\'emethi, A.: Reduction theorem for lattice cohomology,
{\em Int Math Res Notices} {\bf 11} (2015), 2938--2985.

\bibitem[LNN14]{LNN} L\'aszl\'o, T., Nagy, J.  and N\'emethi, A.: Surgery formulae for the Seiberg--Witten invariant
of plumbed 3--manifolds

\bibitem[La72]{Laufer72} Laufer, H.B.: On rational singularities,
{\em Amer. J. of Math.}, {\bf 94}, 597-608, 1972.

\bibitem[Les96]{Lescop} Lescop, C.: Global surgery formula for the Casson--Walker
invariant, {\em Ann. of Math. Studies}  {\bf 140}, Princeton Univ. Press, 1996.

\bibitem[Lim00]{Lim} Lim, Y.: Seiberg--Witten invariants for 3--manifolds in the case $b_1=0$ or $1$,
{\em Pacific J. of Math.} {\bf 195} (2000), no. 1, 179--204.

\bibitem[Li69]{Lipman} Lipman, J.: Rational singularities, with applications to algebraic surfaces
and unique factorization, Inst. Hautes \'Etudes Sci. Publ. Math. {\bf 36} (1969), 195-279.






\bibitem[Kl??]{Kl} Kleiman, St. L.: The Picard scheme,


\bibitem[KN17]{KN} Koll\'ar, J. and N\'emethi, A.:
 Durfee's conjecture on the signature of smoothings of surface singularities,
(an appendix by T. de Fernex), arXiv:1411.1039.
 to appear in {\it Annales Scient. de l'Ecole Norm. Sup.}



\bibitem[N99]{weakly} N\'emethi, A.: ``Weakly'' Elliptic Gorenstein singularities of surfaces,
{\em Inventiones math.},  {\bf 137}, 145-167 (1999).


\bibitem[N99b]{Nfive} N\'emethi, A.: Five lectures on normal surface singularities,
lectures at the Summer School in {\em Low dimensional topology} Budapest,
Hungary, 1998; Bolyai Society Math. Studies {\bf 8} (1999), 269--351.

\bibitem[N05]{NOSZ} N\'emethi, A.: On the Ozsv\'ath--Szab\'o invariant of negative
definite plumbed 3--manifolds,
{\em Geometry and Topology} {\bf 9} (2005), 991--1042.

\bibitem[N07]{trieste} N\'emethi, A.: Graded roots and singularities,
{\em Singularities in geometry and topology},  World
Sci. Publ., Hackensack, NJ (2007), 394--463.

\bibitem[N08]{NPS} N\'emethi, A.: Poincar\'e series associated with surface singularities, in Singularities I, 271--297,
{\em Contemp. Math.} {\bf 474}, Amer. Math. Soc., Providence RI, 2008.

\bibitem[NO09]{NOk} N\'emethi, A. and Okuma, T.:
On the Casson invariant conjecture of Neumann--Wahl, {\em Journal of Algebraic Geometry}
{\bf 18} (2009), 135--149.

\bibitem[N12]{NCL} N\'emethi, A.: The cohomology of line bundles of splice--quotient singularities,
{\em Advances in Math.} {\bf 229} 4 (2012), 2503--2524.

\bibitem[N11]{NJEMS} N\'emethi, A.: The Seiberg--Witten invariants of negative definite plumbed 3--manifolds,
{\em J. Eur. Math. Soc.} {\bf 13} (2011), 959--974.

\bibitem[NN02]{NN1} N\'emethi, A. and Nicolaescu, L.I.:
 Seiberg--Witten invariants and surface singularities,
{\em Geometry and Topology} {\bf 6} (2002), 269--328.

\bibitem[NO09]{NOk} N\'emethi, A. and Okuma, T.:
On the Casson invariant conjecture of Neumann--Wahl, {\em Journal of Algebraic Geometry}
{\bf 18} (2009), 135--149.

\bibitem[NO17]{NO17} N\'emethi, A. and Okuma, T.:
Analytic singularities supported by a specific
integral homology sphere link,
to appear in the
Proceedings dedicated to H. Laufer's 70th birthday (Conference at Sanya, China).




\bibitem[NNI]{NNA1} Nagy, J., N\'emethi, A.:
The Abel map for surface singularities  I. Generalities and  examples,
arXiv 2018.


\bibitem[NNII]{NNA2} Nagy, J., N\'emethi, A.:
The Abel map for surface singularities  II. Generic analytic structure,
 arXiv:1809.03744.


 \bibitem[NN19a]{NNA3} Nagy, J., N\'emethi, A.:
The Abel map for surface singularities  II. Elliptic germs,  arXiv:1902.07493.

\bibitem[NR]{R} Nagy, J:
Invariants of relatively generic structures on normal surface singularities,
arXiv:1910.03275

 \bibitem[NND]{NNAD} Nagy, J., N\'emethi, A.:
The dimension of the image of the Abel map associated with normal surface singularities 
 arXiv:1909.07023



\bibitem[NW90]{NWCasson}
W. Neumann and J.~Wahl, Casson invariants of links of singularities,
Comment. Math. Helvetici {\bf 65} (1990), 58--78.

\bibitem[NW05]{NWsq}
Neumann, W. and Wahl, J.: Complete intersection singularities of splice type as universal abelian covers,
{\em Geom. Topol.} {\bf 9} (2005), 699--755.

\bibitem[NW10]{NWECTh}
W.~D. Neumann and J.~Wahl,
\emph{ The End Curve Theorem for normal complex surface singularities},  J. Eur. Math. Soc. {\bf 12} (2010), 471--503.


\bibitem[Nic04]{Nic04} Nicolaescu, L.: Seiberg--Witten invariants of rational homology $3$--spheres,
{\em Comm. in Cont. Math.} {\bf 6} no. 6 (2004), 833--866.



\bibitem[O04]{OkumaRat} Okuma, T.: Universal abelian covers of rational surface singularities,
{\it Journal of London Mathematical Society} {\bf 70}(2) (2004), 307-324.

\bibitem[O08]{Ok} Okuma, T.: The geometric genus of splice--quotient singularities,
{\em Trans. Amer. Math. Soc.} {\bf 360} 12 (2008), 6643--6659.

\bibitem[O10]{OECTh}
Okuma, T.:
\emph{Another proof of the end curve theorem for normal surface singularities},
J. Math. Soc. Japan {\bf 62}, Number 1 (2010), 1--11.

\bibitem[OWY14]{OWY14} Okuma, T.,   Watanabe Kei-ichi,  Yoshida Ken-ichi:
Good ideals and $p_g$--ideals in two-dimensional normal singularities,
arXiv:1407.1590.

\bibitem[OWY15a]{OWY15a} Okuma, T.,   Watanabe Kei-ichi,  Yoshida Ken-ichi:
Rees algebras and $p_g$--ideals in a two-dimensional normal local domain,
 arXiv:1511.00827.

 \bibitem[OWY15b]{OWY15b} Okuma, T.,   Watanabe Kei-ichi,  Yoshida Ken-ichi:
A characterization of two-dimensional rational singularities via Core of ideals,
arXiv:1511.01553.




\bibitem[Ra72]{Ram} Ramanujam, C.P.: Remarks on Kodaira vanishing theorem, {\it J. Indian Math. Soc.} {\bf 36}
(1972), 41-51.


\bibitem[Re97]{MR}  Reid, M.: Chapters on Algebraic Surfaces.
In: Complex Algebraic Geometry,
IAS/Park City Mathematical Series,  Volume {\bf 3}  (J. Koll\'ar editor),
3-159, 1997.

\bibitem[SzV03]{SzV} Szenes, A. and Vergne, M.: Residue formulae for vector partitions and Euler--Maclaurin sums,
{\em Advances in Appl. Math.} {\bf 30} (2003), 295--342.


\bibitem[Y79]{Yau5} Yau, S. S.-T.: On strongly elliptic singularities,
{\em Amer. J. of Math.}, {\bf 101} (1979), 855-884.

\bibitem[Y80]{Yau1} Yau, S. S.-T.: On maximally elliptic singularities,
{\em Transactions of the AMS}, {\bf 257} Number 2 (1980), 269-329.

\end{thebibliography}
\end{document}